\documentclass[11]{article}
\usepackage{amsfonts, amssymb, latexsym, amsthm, amsmath, verbatim,hyperref,enumerate}
\usepackage{authblk,fullpage,color}

\usepackage{float}

\usepackage{graphicx}
\theoremstyle{plain}
\newtheorem{theorem}{Theorem}
\newtheorem{corollary}[theorem]{Corollary}
\newtheorem{lemma}[theorem]{Lemma}
\newtheorem{proposition}[theorem]{Proposition}

\theoremstyle{definition}
\newtheorem{definition}[theorem]{Definition}
\newtheorem{example}[theorem]{Example}

 \newtheorem{algorithm}[theorem]{Algorithm}
\newtheorem{problem}{Problem}

\theoremstyle{remark}

\def\F{F}
\def\G{G}

\def\B{\mathcal{C}}
\def\b{\beta}
\def\BS{\mathcal{B}}
\def\D{\mathcal{D}}
\def\t{\mathfrak{i}}
\def\C{\mathcal{A}}
\def\S{\mathfrak{S}}
\def\ST{R}
\def\TB{T}

\def\STT{R}
\def\Pn{P(S;n)}
\def\d{\delta}
\def\a{\alpha}

\newcommand{\ATK}{{P(S;n)^{\nearrow k}}}
\newcommand{\DTK}{{P(S;n)_{\searrow k}}}
\newcommand{\cielo}{{\overline{P(S;n)}}}
\newcommand{\infierno}{{\underline{P(S;n)}}}
\newcommand{\cielovacio}{{\overline{P(\emptyset;n)}}}
\newcommand{\infiernovacio}{{\underline{P(\emptyset;n)}}}
\newcommand{\PiAK}{{\Pi^{\nearrow k}}}
\newcommand{\PiA}{{\Pi^{\nearrow}}}

\newcommand{\PiDK}{{\Pi_{\searrow k}}}

\title{Peak Sets of Classical Coxeter Groups}
 \author[1]{Alexander Diaz-Lopez\thanks{\textcolor{blue}{\href{mailto:adiaz4@nd.edu}{adiaz4@nd.edu}}}}
 \author[2]{Pamela E. Harris\thanks{\textcolor{blue}{\href{mailto:Pamela.Harris@usma.edu}{Pamela.Harris@usma.edu}}. This research
was performed while the author held a National Research Council Research Associateship Award at USMA/ARL.}}
 \author[3]{Erik Insko\thanks{\textcolor{blue}{\href{mailto:einsko@fgcu.edu}{einsko@fgcu.edu}}}}
 \author[4]{Darleen Perez-Lavin \thanks{\textcolor{blue}{\href{mailto:darleenpl@uky.edu}{darleenpl@uky.edu}}}}
 \affil[1]{Department of Mathematics, University of Notre Dame}
 \affil[2]{Department of Mathematical Sciences, United States Military Academy}
 \affil[3]{Department of Mathematics, Florida Gulf Coast University}
 \affil[4]{Department of Mathematics, University of Kentucky}
 
\date{ }

\begin{document}
%
%
%
%
%
%
%
%
%
%
%
%
%
%


\maketitle
 \begin{abstract} \noindent
We say a permutation $\pi=\pi_1\pi_2\cdots\pi_n$ in the symmetric group $\S_n$ has a \emph{peak} at index $i$ if
$\pi_{i-1}<\pi_i>\pi_{i+1}$ and we let
$P(\pi)=\{i \in \{1, 2, \ldots, n\}  \, \vert \, \mbox{$i$ is a peak of $\pi$}\}$. Given a set $S$ of positive integers, we let
$P (S; n)$ denote the subset of $\S_n$ consisting of all permutations $\pi$, where $P(\pi) =S$.
In 2013, Billey, Burdzy, and Sagan proved $|P(S;n)| = p(n)2^{n-\lvert S\rvert-1}$, where $p(n)$ is a polynomial of degree
$\max(S)
- 1$.
In 2014, Castro-Velez et al. considered the Coxeter group of type $B_n$ as the group of signed permutations on $n$ letters and
showed that $\lvert P_B(S;n)\rvert=p(n)2^{2n-|S|-1}$ where $p(n)$ is the same polynomial of degree $\max(S)-1$. In this paper we
partition the sets $P(S;n) \subset \S_n$
studied by Billey, Burdzy, and Sagan into subsets of $P(S;n)$ of permutations with peak set $S$ that end with an ascent to a
fixed
integer $k$ or a descent
and provide polynomial formulas for the cardinalities of these subsets.
After embedding the Coxeter groups of Lie type $C_n$ and $D_n$ into $\S_{2n}$, we partition these groups into bundles of
permutations
$\pi_1\pi_2 \cdots\pi_n|\pi_{n+1}\cdots \pi_{2n}$ such that $\pi_1\pi_2\cdots \pi_n$ has the same relative order as some
permutation $\sigma_1\sigma_2\cdots\sigma_n \in \S_n$.
This allows us to count the number of permutations
in types $C_n$ and $D_n$ with a given peak set $S$ by reducing the enumeration to calculations in the symmetric group and sums
across the rows of Pascal's triangle.
\end{abstract}


\section{Introduction}

We say a permutation $\pi=\pi_1\pi_2\ldots \pi_n$ in the symmetric group $\S_n$ has a \emph{peak} at index $i$ if
$\pi_{i-1}<\pi_i>\pi_{i+1}$.  We let $[n]:=\{1,2,\ldots, n\}$
and define the peak set of a permutation $\pi$ to be the set of peaks in $\pi$: \[ P(\pi) = \{i \in [n]\, \vert \, \mbox{ $i$ is
a
peak of $\pi$}\}. \]
Given a subset $S \subset [n],$
 we denote
 the set of all permutations with peak set $S$ by
\[P(S;n) = \{ \pi \in \S_n \,  \vert \,  P(\pi) = S\}.\]
We say a set $S \subset [n]$ is \emph{$n$-admissible} (or simply \emph{admissible} when $n$ is understood) provided $P(S;n)\neq
\emptyset$.

While the combinatorics of Coxeter groups has fascinated mathematicians for generations \cite{BB00},
the combinatorics of peaks has only recently caught the eye of the mathematical community.
Stembridge was one of the first to study the combinatorics of peaks. In 1997, he defined a peak analog of Stanley's theory of
poset partitions \cite{S97}.
In 2003, Nyman showed that taking formal sums of permutations according to their peak sets
gives a non-unital subalgebra of the group algebra of the symmetric group \cite{N03}.
This motivated several papers studying peak (and descent) algebras of classical Coxeter groups
\cite{ABN04, ANO06,BH06,P07}.
Peaks have also been linked to the Schubert calculus of isotropic flag manifolds~\cite{BMSV02, BS02,BH95} and the generalized
Dehn-Sommerville equations~\cite{ABS06,BMSV00, BHV03}.

In 2013, Billey, Burdzy, and Sagan counted the number of elements in the sets $P(S;n)$.
For any $n$-admissible set $S$, they found these cardinalities satisfy
\begin{align}
| P(S;n) |&= p(n) 2^{n-| S|-1}\label{polyeq}
\end{align}
where $|S|$ denotes the cardinality of the set $S$, and where the
\emph{peak polynomial} $p(n)$
is a polynomial of degree $\max(S)-1$ that takes integral values when evaluated at integers \cite[Theorem
1.1]{BBS13}.
Their study was motivated by a problem in probability theory which explored the mass distribution on graphs as it relates to
random
permutations with specific peak sets; this research was presented in \cite{BBPS14}.
Billey, Burdzy, and Sagan also computed closed formulas for the peak polynomials $p(n)$ for various special cases of $P(S;n)$
using the method of finite differences, and
Billey, Fahrbach, and Talmage then studied the coefficients and zeros of peak polynomials \cite{BFT14}.

Shortly after Billey, Burdzy, and Sagan's article appeared on the arXiv, Kasraoui proved one of their open conjectures and
identified the most probable peak set for a random permutation \cite{K14}.
Then Castro-Velez et al. generalized the work of Billey, Burdzy, and Sagan to study peak sets of type $B$ signed permutations
\cite{CV14}.
They studied two sets $P_B(S;n)$ and $\hat{P}_B(S;n)$ of signed permutations with peak set $S$, whose formal definition we
introduce in Subsection  \ref{comparison}.
Their main result regarding the set $P_B(S;n)$ used induction to prove (\cite[Theorem 2.4]{CV14})
\begin{align}
|P_B(S;n)| &= |P(S;n)|2^n=p(n) 2^{2n-| S|-1}. \label{Bpolyeq}
\end{align}
 Note that $p(n)$ is the same polynomial as that of Equation \eqref{polyeq}. 

Motivated by extending the above mentioned results to other classical Coxeter groups, our work begins by partitioning the sets $P(S;n)$ studied by Billey, Burdzy, and Sagan
into subsets $\ATK$ and $\DTK$ of permutations ending with an ascent or a descent to a fixed
 $k$,
respectively. With these partitions on hand, we show  in Theorems \ref{Pv1} and \ref{Pv2} that
the cardinalities of these sets are governed by polynomial formulas similar to those discovered by Billey, Burdzy, and Sagan.
These results are presented in Section \ref{PartitionPsn}.

We then embed the Coxeter groups of type $C_n$ and $D_n$ into $\S_{2n}$ and call these embedded subgroups
$\B_n, \D_n \subset \S_{2n}$
the \emph{mirrored permutations} of types $C_n$ and $D_n$ respectively (Section \ref{partitionbd}).
For each $\pi \in \S_{n}$ we define the
 \emph{pattern bundle} of $\pi$ in types $C_n$ and $D_n$ in Definitions \ref{relorder} and \ref{relorderD}.
Each \emph{pattern bundle} consists of permutations $\tau_1\tau_2\cdots \tau_n| \tau_{n+1}\cdots \tau_{2n}$ such that  $\tau_1
\tau_2 \cdots \tau_n$
flattens to $\pi_1\pi_2 \cdots \pi_n$, meaning $\tau_1\tau_2\cdots\tau_n$ has the same relative order as $\pi_1\pi_2\cdots\pi_n$.
These pattern bundles have the following properties: 1) they partition the groups $\B_n$ and $\D_n$; 2) they are
indexed by the elements of $\S_n$, and; 3) they have
size $2^n$ in $C_n$ and $2^{n-1}$ in $D_n$. 
This process allows us to give concise proofs of the following two identities (Theorem \ref{HIPB} \eqref{HIPB1} and \eqref{HIPB2}, respectively):
\[| P_C(S;n)|  = p(n) 2^{2n-| S|-1}\hspace{.2in}\mbox{and}\hspace{.2in}
| P_D(S;n)|  = p(n) 2^{2n-| S|-2}.\]
We note that the polynomial appearing above is the same as that of Equation \eqref{polyeq}. Moreover the proof
of Theorem
\ref{HIPB} (\ref{HIPB1}) is much shorter than the one given by \cite[Theorem 2.4]{CV14}, and Theorem \ref{HIPB} (\ref{HIPB2}) has not appeared before in the
literature.

Finally in Section \ref{typeCD} we prove our main result, Theorem \ref{key}.
 We use the formulas for  $|\ATK|$ and $|\DTK|$ from Section
\ref{PartitionPsn}
and
 sums of binomial coefficients to enumerate the set of permutations with peak set $S \subset [n]$ in $C_n$ and $D_n$.

We end this introduction with a remark on the history of this collaboration.
The last three authors of this article began their study of peak sets in classical Coxeter groups before Castro-Velez et al. had
published their results from
type $B_n$, and focused their study on the Coxeter (Weyl) groups of types $C_n$ and $D_n$ using presentations of these groups
described in Billey and Lakshmibai's text on \cite[pp. 29,34]{BL00}.   
While Perez-Lavin was presenting the preliminary results of this paper at the USTARS 2014
conference held at UC-Berkeley, we met Alexander Diaz-Lopez who told us of his recently completed work with Castro-Velez et al.
\cite{CV14}.
Knowing that the Coxeter groups of types $B$ and $C$ are isomorphic, we were immediately intrigued
to see what connections could be found between the two works.
We were delighted to find that we used vastly different techniques to count the elements of $P_B(S;n)$ and $P_C(S;n)$,
and discovered an isomorphism between the two groups which preserves peak sets (up to a reordering of the peaks).
We highlight these connections and compare and contrast the two works in Subsection \ref{comparison}.

\section{Partitioning the set \texorpdfstring{$P(S;n)$}{P(S;n)}} \label{PartitionPsn}



To make our approach precise, we begin by setting notation and giving some definitions.

\begin{definition}\label{alpha-delta}
For a given peak set $S\subset [n-1]$, we define
\begin{align*}
\ATK&:=\{\pi\in P(S;n) \ \mid \ \pi_{n-1}<\pi_n\mbox{ and }\pi_n=k\},\\
\cielo&:=\mathop{\sqcup}_{k=1}^n\ATK,\\
\DTK&:=\{\pi\in P(S;n) \ \mid \ \pi_{n-1}>\pi_n\mbox{ and }\pi_n=k\},\mbox{ and}\\
\infierno&:=\mathop{\sqcup}_{k=1}^n \DTK.
\end{align*}
\end{definition}
We remark that $P(S;n)^{\nearrow 1}=\emptyset$ because a permutation cannot end with an ascent to 1.
Similarly $P(S;n)_{\searrow n}=\emptyset$ since a permutation cannot end with a descent to $n$. Therefore
the sets $\cielo$ and $\infierno$ are the following disjoint unions of
sets
\[\cielo=\mathop{\sqcup}_{k=2}^n\ATK\mbox{    and   }
\infierno=\mathop{\sqcup}_{k=1}^{n-1} \DTK.\]
Since every $\pi\in P(S;n)$ either ends with an ascent or a descent we see \[ P(S;n)=\cielo\sqcup
\infierno. \]
Our next lemma counts the permutations without peaks that end with an ascent to $k$.
\begin{lemma}\label{powerof2}
If $2\leq k\leq n$, then $|P(\emptyset;n)^{\nearrow k}|=2^{k-2}$.
\end{lemma}
\begin{proof}
Let $2\leq k\leq n$ and suppose that $\pi=\pi_1\pi_2\cdots\pi_n\in P(\emptyset;n)^{\nearrow k}$. Hence
$P(\pi)=\emptyset$ and
$\pi_{n-1}<\pi_n=k$.
Let us further assume that $\pi=\tau_A \ 1\ \tau_B \ k$, where $\tau_A$ and $\tau_B$ are the portions of $\pi$ to the left and
right of $1$, respectively.
Since $P(\pi)=\emptyset$ we know $\tau_A$ must decrease, while $\tau_B$ must increase.
However the values of $\tau_B$ must come from the set $\{2,3,\ldots,k-1\}$ because $\pi_{n-1}<\pi_{n}=k$, and there is one $\pi
\in P(\emptyset;n)^{\nearrow k}$
for each subset of $\{2,3,\ldots,k-1\}$ as such a $\pi$ is completely determined by which elements from that set appear in
$\tau_B$.
Hence we see $|P(\emptyset;n)^{\nearrow k}|=2^{k-2}$.
\end{proof}

We will next prove a recursive formula for the number of permutations with specified peak set $S$ that end in an ascent to a
fixed integer $k$.

\begin{lemma}\label{ascenttok}
Let $S\subset [n-1]$ be a nonempty admissible set. Let $m=\max(S)$ and fix an integer $k$, where $1\leq k\leq n$. If $S_1=S
\setminus \{m\}$ and $S_2=S_1\cup\{m-1\}$, then
\[\lvert \ATK \rvert= \sum_{i=0}^{k-2} \binom{k-1}{i} \binom{n-k}{m-i-1} \lvert P(S_1;m-1)\rvert  2^{k-i-2} -\lvert
P(S_1;n)^{\nearrow k}\rvert-\lvert P(S_2;n)^{\nearrow k}\rvert.\]
\end{lemma}
\begin{proof}
Observe that if $k=1$, then the result holds trivially as all terms in the statement are identically zero.  Let $2\leq k \leq n$ and
denote $ \PiAK$ as the set of permutations ending with an ascent to $k$ that have peak set $S_1$ in the first $m-1$ spots and
no peaks in the last $m-n+1$ i.e.,
\[\PiAK=\{\pi\in \S_n \ \mid  P(\pi_1\pi_2\cdots \pi_{m-1})=S_1, P(\pi_{m}\cdots \pi_{n} ) = \emptyset, \text{ and }
\pi_{n-1}<\pi_n=k\}.\]
We compute the cardinality of the set $\PiAK$ by counting the number of ways to construct a permutation in
$\PiAK$.

First we select a subset $P_1= \{ \pi_1, \pi_2, \ldots, \pi_{m-1} \} \subset [n] \setminus \{k\}$ (as we fix $\pi_n$ to be $k$).
When selecting $P_1$,
we can choose $i$ numbers from $\{1, 2, \ldots, k-1 \}$ to include in $P_1$  for each $0 \leq i \leq k-1$ and then choose the
remaining $m-i-1$ numbers from the set $\{ k+1,k+2, \ldots, n \}$ to fill the remainder of $P_1$.
Thus there are $\binom{k-1}{i}\cdot \binom{n-k}{m-i-1}$ ways to select the elements of $P_1$. By definition, there are $|P(S_1,
m-1)|$ ways to arrange the $m-1$ elements of $P_1$ into a permutation $\pi_1\pi_2\cdots \pi_{m-1}$ satisfying $P(\pi_1\pi_2\cdots
\pi_{m-1})=S_1$.

Let $P_2 = \{ \pi_m,\pi_{m+1}, \ldots, \pi_n\}=[n]\setminus P_1$, where $\pi_n = k$.  There are $n-(m-1)=n-m+1$ numbers in $P_2$,
and there are precisely $k-i-1$ elements from the set $\{1, 2, \ldots, k-1 \}$
that were not chosen to be part of $P_1$.  That means $k$ is the $(k-i)^\text{th}$ largest integer in the set $P_2$. By
flattening
the numbers in $P_2$, we can see there are $|P(\emptyset;n-m+1)^{{\nearrow k-i}}|$ ways to arrange the elements
of
$P_2$ to create
a subpermutation $\pi_m\pi_{m+1}\cdots \pi_n$ that satisfies
\[ P(\pi_{m}\cdots \pi_{n} ) = \emptyset \text{ and }  \pi_{n-1}<\pi_n=k .\]

By Lemma \ref{powerof2} we know that $|P(\emptyset;n-m+1)^{{\nearrow k-i}}|=2^{k-i-2}$ when $k-i \geq 2$ and it
is
$0$ otherwise.
Of course $k-i\geq 2$ when $i \leq k-2$.
Putting this all together, we see that the number of ways to create a permutation in $\PiAK$ is $\sum_{i=0}^{k-2}
\binom{k-1}{i} \binom{n-k}{m-i-1} \lvert P(S_1;m-1)\rvert
2^{k-i-2}$, or in other words
\begin{equation} \lvert \PiAK \rvert =\sum_{i=0}^{k-2} \binom{k-1}{i} \binom{n-k}{m-i-1} \lvert P(S_1;m-1)\rvert  2^{k-i-2}
\label{amazing}. \end{equation}

Next we consider a different way to count the elements of $\PiAK$. Note that we have not specified whether $\pi_{m-1} >
\pi_m$ or $\pi_{m-1} < \pi_m$.
So in particular based on the definition of $\PiAK$ and its restrictions on $P(\pi_1\pi_2\cdots\pi_{m-1})$ and
$P(\pi_{m}\pi_{m+1}\cdots\pi_n)$, all of the following are possible:
\[ P(\pi)=S, P(\pi)=S_1, \text{ or } P(\pi)=S_2, \text{ for } \pi\in\PiAK.\]
Hence \[\PiAK=\ATK\sqcup P(S_1;n)^{{\nearrow k}}\sqcup P(S_2;n)^{{\nearrow k}}.\]
Thus
\begin{align}
\lvert \PiAK \rvert &=\lvert \ATK\rvert +\lvert P(S_1;n)^{{\nearrow k}}\rvert +\lvert
P(S_2;n)^{{\nearrow k}} \rvert.\label{eq12}
\end{align}
The result follows from setting Equations \eqref{amazing} and \eqref{eq12} equal to each other and solving for the quantity
$\lvert \ATK \rvert.$
\end{proof}

The following lemma will be used in the proofs of Lemmas \ref{ascent} and \ref{descent}.
\begin{lemma} \label{emptyalphadelta}
If $n \geq 2$ then
\begin{itemize}
 \item the cardinality of $|\infiernovacio| = 1$, and
\item the cardinality of $|\cielovacio| = 2^{n-1}-1$.
\end{itemize}
\end{lemma}

\begin{proof}
The only permutation $\pi \in P(\emptyset;n)$ that ends in a descent is $ n=\pi_1>\pi_2> \cdots > \pi_n=1,$ therefore $|\infiernovacio| = 1$.
On the other hand, it is easy to see that $P(\emptyset;n) = 2^{n-1}$ as Billey, Burdzy, and Sagan proved in \cite[Proposition
2.1]{BBS13}.
Since $P(\emptyset;n) = \cielovacio \sqcup \infiernovacio$ we
compute \[ |\cielovacio| =
|P(\emptyset; n)|- |\infiernovacio|   =   2^{n-1}-1. \qedhere \]
\end{proof}

The following result allows us to recursively enumerate the set of permutations with specified peak set $S$  that end with an
ascent.

\begin{lemma}\label{ascent}Let $S\subset [n-1]$ be a nonempty $n$-admissible set, and let $m=\max(S)$.
If we let $S_1=S\setminus\{m\}$ and $S_2=S_1\cup\{m-1\}$, then
\[ \lvert \cielo\rvert =\binom{n}{m-1}\left(2^{n-m}-1\right)\lvert P(S_1;m-1)\rvert -\lvert
{\overline{P(S_1;n)}} \rvert-\lvert
{\overline{P(S_2;n)}} \rvert .\]
\end{lemma}

\begin{proof}
Let $S\subset [n-1]$ be an admissible set with $m=\max(S)$.
Define the sets $S_1=S\setminus\{m\},$ $ S_2=S_1\cup\{m-1\}$ and \[\PiA=\{\pi\in \S_n \ \mid \ P(\pi_1\pi_2\cdots
\pi_{m-1})=S_1, P(\pi_{m}\cdots\pi_n)=\emptyset \text{ and }  \pi_{n-1}<\pi_{n}\}.\]
Next we compute the cardinality of the set $\PiA$.
To do so, we observe that there are $\binom{n}{m-1}$ choices for the values of $\pi_{1},\ldots,\pi_{m-1}$, and by definition,
there are $\lvert P(S_1;m-1)\rvert$ ways to arrange the values of
$\pi_{1},\ldots,\pi_{m-1}$  so that $P(\pi_1\pi_2\cdots\pi_{m-1}) = S_1$.
Once we have chosen the values of $\pi_1, \pi_2, \ldots, \pi_{m-1}$, the values of \[ \pi_m, \pi_{m+1}, \pi_{m+2}, \ldots,
\pi_n\]
are determined.
We note that there are $|{\overline{P(\emptyset;n-m+1)}}|$ ways to arrange the values of $\pi_m, \ldots,\pi_n$,
so
that
$P(\pi_m\cdots\pi_n)=\emptyset$ and $\pi_{n-1}<\pi_n$.

Yet Lemma \ref{emptyalphadelta} proved that $|{\overline{P(\emptyset;n-m+1)}}|=2^{n-m}-1$. Hence we see that
\begin{align}
\lvert \PiA \rvert &=\binom{n}{m-1}\left(2^{n-m}-1\right)\lvert P(S_1;m-1)\rvert .\label{eq3}
\end{align}

On the other hand  $\PiA=\cielo\sqcup {\overline{P(S_1;n)}}\sqcup
{\overline{P(S_2;n)}}$ by the defining conditions of $\PiA$.
Hence
\begin{align}
\lvert \PiA \rvert&=\lvert \cielo\rvert +\lvert {\overline{P(S_1;n)}}\rvert
+\lvert
{\overline{P(S_2;n)}} \rvert.\label{eq4}
\end{align}
When we set the right hand sides of Equations (\ref{eq3}) and (\ref{eq4}) equal to each other and solve for $\lvert
\cielo\rvert$ we see that
\[ \lvert \cielo\rvert =\binom{n}{m-1}\left(2^{n-m}-1\right)\lvert P(S_1;m-1)\rvert -\lvert
{\overline{P(S_1;n)}} \rvert-\lvert
{\overline{P(S_2;n)}} \rvert . \qedhere \]
\end{proof}

The following examples illustrate the recursion used to prove Lemmas \ref{ascenttok} and \ref{ascent}.

\begin{example}
We make use of Lemma \ref{ascenttok} to compute $\lvert P(\{3\};5)^{{\nearrow
3}}\rvert$. Let $S$ be the set
$S=\{3\}\subset[5]$. Note that $m=\max(S)=3$. Then we compute
	\begin{equation}\label{eq10}
		\lvert P(\{3\};5)^{{\nearrow
3}}\rvert=\left(\binom{2}{0}\binom{2}{2}2^{1}+\binom{2}{1}\binom{2}{1}2^{0}\right)\lvert
P(\emptyset;2)\rvert -\lvert P(\emptyset;5)^{{\nearrow 3}} \rvert-\lvert P(\{2\};5)^{{\nearrow
3}}\rvert.
	\end{equation}
Some small computations show that $P(\emptyset; 2) = \{12, 21\} $, $P(\emptyset; 5)^{{\nearrow 3}} =
\{54213, 54123\} $, and
	\[P(\{2\}; 5)^{{\nearrow 3}} =  \left \{ 45213, 25413, 45123, 15423 \right \}.\]
 Accordingly, we can see that Equation \eqref{eq10} gives
	\[ \lvert P(\{3\};5)^{{\nearrow 3}}\rvert=(2+4)(2)-2-4=6.\]
\end{example}
\begin{example}
In this example we make use of Lemma \ref{ascent} to compute $\lvert  {\overline{P(\{3\};5)}}\rvert$. If we let
$S=\{3\} \subset [5]$ then
$m=\max(S)=3$. We then have
	\begin{align}
 		\lvert {\overline{P(\{3\};5)}}\rvert =\binom{5}{2}\left(2^{5-3}-1\right)\lvert
P(\emptyset;2)\rvert -\lvert
{\overline{P(\emptyset;5)}} \rvert-\lvert {\overline{P(\{2\};5)}} \rvert. \label{eqex2}
	\end{align}
Some small computations show
\begin{align*}
P(\emptyset; 2) &= \{12 , 21\} \\
{\overline{P(\emptyset; 5)}}&=\left \{ \begin{matrix} 54321, 54213, 54123, 53214, 53124, 52134, 51234,43215,\\
43125, 42135, 32145, 41235,
31245, 21345, 12345\end{matrix} \right \}.
\end{align*}
Direct computations yield
\[{\overline{P(\{2\}; 5)}}=\left \{ \begin{matrix}
45312, 35412, 45213, 25413, 45123,15423, 35214,25314, 35124, 25134,\\ 15324, 15234, 34215, 24315, 34125, 24135, 23145, 14325,
14235, 13245
\end{matrix} \right \}.\]
Equation $\eqref{eqex2}$ gives
	\[\lvert {\overline{P(\{3\};5)}}\rvert =\binom{5}{2}\left(2^{5-3}-1\right)(2) -15-20=25.\]
\end{example}

Next we consider permutations that end in a descent to a specific value $k$.

\begin{lemma}\label{descenttok}Let $S\subset [n-1]$ be a nonempty admissible set, $m=\max(S)$, and fix an integer $k$, where
$1\leq k\leq n$. If $S_1=S\setminus\{m\}$ and $S_2=S_1\cup\{m-1\}$, then
\[\lvert  \DTK\rvert=\binom{n-k}{n-m}\lvert  P(S_1;m-1)\rvert-\lvert
P(S_1;n)_{{\searrow k}}\rvert-\lvert  P(S_2;n)_{{\searrow k}}\rvert.\]
\end{lemma}

\begin{proof}
Let $k$ be a fixed integer between $1$ and $n$. Let $S\subset [n-1]$ with $m=\max(S)$. We let $S_1=S\setminus\{m\}$ and
$S_2=S_1\cup\{m-1\}$. We let
\[\PiDK=\{\pi\in \S_n \ \mid P(\pi_1\pi_2\cdots \pi_{m-1})=S_1,\; \pi_{m}>\pi_{m+1}>\cdots>\pi_{n-1}>\pi_n=k\}.\]
We want to compute the cardinality of the set $\PiDK$. To do so, we observe that the values of $\pi_{m},\ldots,\pi_{n-1}$
must be larger than $k$. There are $n-k$ possible values, namely the values $k+1,k+2,\ldots,n$. Hence we can select $n-m$ values
from $n-k$ possible options, which gives us $\binom{n-k}{n-m}$ choices.
Then this determines the remaining values $\pi_1,\ldots,\pi_{m-1}$. By definition $\lvert  P(S_1;m-1)\rvert$ gives the number of
elements with peak set $S_1$. Hence
\begin{align}
\lvert  \PiDK\rvert&=\binom{n-k}{n-m}\lvert  P(S_1;m-1)\rvert.\label{eq5}
\end{align}
Notice that by construction $\PiDK=\DTK\sqcup P(S_1;n)_{{\searrow k}}\sqcup
P(S_2;n)_{{\searrow k}}.$
Hence the cardinality of $\PiDK$ is also given by
\begin{align}
\lvert  \PiDK\rvert&=\lvert  \DTK\rvert+\lvert  P(S_1;n)_{{\searrow
k}}\rvert+\lvert  P(S_2;n)_{{\searrow k}}\rvert.\label{eq6}
\end{align}
The result follows from setting Equations \eqref{eq5} and \eqref{eq6} equal to each other and solving for the quantity $\lvert
\DTK\rvert.$
\end{proof}
The following result allows us to recursively enumerate the set of permutations with specified peak set $S$  that end with a
descent.

\begin{lemma} \label{descent}
Let $S\subset [n-1]$ be a nonempty admissible set, $m=\max(S)$. If $S_1=S\setminus\{m\}$ and $S_2=S_1\cup\{m-1\}$, then
\[\lvert  \infierno\rvert=\binom{n}{m-1}\lvert  P(S_1;m-1)\rvert-\lvert
{\underline{P(S_1;n)}}\rvert-\lvert  {\underline{P(S_2;n)}}\rvert.\]
\end{lemma}

\begin{proof}
{By Definition \ref{alpha-delta}, 
\[ |\infierno|
 = \sum_{k=1}^{n-1}|\DTK|. \] }
 Using this equation and Lemma \ref{descenttok} we get
	\begin{align*}
		 |\infierno| & =\sum_{k=1}^{n-1}\left[ \binom{n-k}{n-m}\lvert  P(S_1;m-1)\rvert-\lvert
P(S_1;n)_{{\searrow k}}\rvert-\lvert  P(S_2;n)_{{\searrow k}}\rvert\right] & \\ &
=\binom{n}{m-1}|P(S_1; m-1)| - |{\underline{P(S_1; n)}}|-|{\underline{P(S_2; n)}}|,
	\end{align*}
where the last equality comes from the identity $\sum^n_{k=0} \binom{k}{c} = \binom{n+1}{c+1}$.
\end{proof}

As before, we provide an example that illustrates the use of the previous results.

\begin{example}Consider the set $S=\{3\}\subset[5]$, hence $m=\max(S)=3$. We want to compute $\lvert
P(\{3\};5)_{{\searrow 2}}\rvert$. By
Lemma \ref{descenttok} we have that
\begin{align*}
\lvert  P(\{3\};5)_{{\searrow 2}}\rvert&=\binom{3}{2}\lvert  P(\emptyset;2)\rvert-\lvert
P(\emptyset;5)_{{\searrow 2}}\rvert-\lvert
P(\{2\};5)_{{\searrow 2}}\rvert.
\end{align*}
Some simple computations show that \[P(\emptyset; 2) = \{12 , 21\} ,\ P(\emptyset; 5)_{{\searrow 2}}= \emptyset
,\mbox{ and }P(\{2\}; 5)_{{\searrow 2}} =  \{15432\}.\]
Therefore
\[\lvert  P(\{3\};5)_{{\searrow 2}}\rvert=3(2)-0-1=5.\]
In fact,  $P(\{3\};5)_{{\searrow 2}}=\{51432, 41532, 31542, 14532, 13542\}.$

We want to compute $|{\underline{P(\{3\};5)}}|$. By Lemma \ref{descent} we have that
\begin{align*}
\lvert  {\underline{P(\{3\};5)}}\rvert&=\binom{5}{2}\lvert  P(\emptyset;2)\rvert-\lvert
{\underline{P(\emptyset;5)}}\rvert-\lvert
{\underline{P(\{2\};5)}}\rvert.
\end{align*}
Again we can compute that
\[P(\emptyset; 2) = \{12 , 21\} ,\ {\underline{P(\emptyset; 5)}} = \{ 54321 \} ,\mbox{ and }
{\underline{P(\{2\}; 5)}} =  \{ 45321, 35421, 25431,
15432  \}.\]
Thus
\[\lvert  {\underline{P(\{3\};5)}}\rvert=10(2)-1-4=15.\]
In fact,  ${\underline{P(\{3\};5)}}=\left \{ \begin{matrix}53421, 43521, 34521, 52431, 42531, 32541, 24531,\\
23541, 51432, 41532,31542,
21543, 14532, 13542, 12543 \end{matrix}\right\}.$
\end{example}

The following two theorems allow us to easily calculate closed formulas for $\lvert \infierno \rvert $
and $\lvert P(S;n)^{{\nearrow k}} \rvert $ using
the method of finite differences  \cite[Proposition 1.9.2]{S12}.  We start by applying Lemma \ref{descent} in an induction
argument to show $|\infierno|$ is given by a polynomial $p_\d(n)$.

\begin{theorem}\label{Pv1}
 Let $S\subset [n-1]$ be an admissible set. If $S= \emptyset$ take $m=1$; otherwise let $m=\max(S)$.
 Then the cardinality of the set $\infierno$ is given by
 \[ \lvert \infierno \rvert = p_\d(n)\]
 where $p_\d(n)$ is a polynomial in the variable $n$ of degree $m-1$ that returns integer values for all integers $n$.
\end{theorem}

\begin{proof}
 We induct on the sum $\t =\sum_{i \in S} i$. {When $\t=0$ the set $S$ is empty. By Lemma  \ref{emptyalphadelta} we get $|\infiernovacio |=1$, and so $p_{\delta}(n)=1$ is a polynomial of degree $0$.}

 Let $S\subset [n-1]$ be non-empty, with $m=\max(S)$ and $\sum_{i \in S} i = \t$.
 Let $S_1=S\setminus\{m\}$ and $S_2=S_1\cup\{m-1\}$, and
 note, in particular, that the sums $\sum_{i \in S_1} i $ and $\sum_{i \in S_2} i$ are both strictly less than $\t$.  By induction, we know that $|{\underline{P(S_1;n)}}|=p_{\d_1}(n) \text{ and } |{\underline{P(S_2;n)}}|=p_{\d_2}(n)$, where $p_{\d_1}$ and $p_{\d_2}$ are polynomials of degrees less than $m-1$ that have integral values when
evaluated at integers. 

 {By Equation \eqref{polyeq} we have
 $|P(S_1;m-1)|= p(m-1)2^{(m-1)-|S_1|-1}$ and this expression
 returns an integer value when evaluated at any integer $m-1$ \cite[Theorem 2.2]{BBS13}.}  Since the expression $p(m-1)2^{(m-1)-|S_1|-1}$ is an integer-valued constant with respect to $n$, we see that
$p(m-1)2^{(m-1)-|S_1|-1}\binom{n}{m-1}$
is a polynomial in the variable $n$ of degree $m-1$. These facts, together with Lemma \ref{descent} imply
 \begin{align*}
 \lvert  \infierno\rvert&=\binom{n}{m-1}\lvert  P(S_1;m-1)\rvert-\lvert
{\underline{P(S_1;n)}}\rvert-\lvert  {\underline{P(S_2;n)}}\rvert \\
 				&=\binom{n}{m-1} p(m-1)2^{(m-1)-|S_1|-1}  -p_{\d_1}-p_{\d_2} \\
				&=p_\d
\label{short1}
 \end{align*}
  where $p_\d$ is a polynomial in the variable $n$ of degree $m-1$ that has integer values when evaluated at integers.
\end{proof}

Using Lemma \ref{ascenttok}, we show $\lvert \ATK \rvert$ is given by a polynomial.

\begin{theorem}\label{Pv2}
Let $S\subset [n-1]$ be an admissible set. If $S= \emptyset$ take $m=1$; otherwise let $m=\max(S)$. Fix an integer $k$ satisfying
$2\leq k\leq n$, then the cardinality of the set $\ATK$ is given by
\[\lvert \ATK \rvert= p_{\a(k)}(n) \] where $p_{\a(k)}(n)$ is a polynomial of degree $m-1$ that returns integer values for
all integers $n$.
\end{theorem}

\begin{proof}
We proceed by induction on the sum $\t = \sum_{i \in S} i$.    When $\t=0$ the set $S$ is empty. By Lemma  \ref{powerof2} we get $|P(\emptyset;n)^{{\nearrow k}} |= 2^{k-2}$, which is a polynomial of degree $0$ in the indeterminate $n$. 


 Consider a non-empty subset $S\subset [n-1]$ with $m=\max(S)$ and $\sum_{i \in S} i = \t$.
 Let $S_1=S\setminus\{m\}$ and $S_2=S_1\cup\{m-1\}$, and
 note in particular, that the sums $\sum_{i \in S_1} i $ and $\sum_{i \in S_2} i$ are both strictly less than $\t$. By induction, we know that $| P(S_1;n )^{{\nearrow k}} | =p_{\a_1(k)}(n)$ and
$|P(S_2;n)^{{\nearrow k}}|=p_{\a_2(k)}(n)$
where $p_{\a_1(k)}(n)$ and $p_{\a_2(k)}(n)$ are each polynomials of degrees less than $m-1$ that have integer values when
evaluated at integers.

{By Equation \eqref{polyeq} we know $|P(S_1;m-1)|=p(m-1)2^{(m-1)-|S_1|-1}$ is an integer-valued function when
evaluated at any integer $m-1$}, and it is a constant function with respect to $n$.
Hence the expression $\binom{k-1}{i}  \lvert P(S_1;m-1)\rvert  2^{k-i-2}$ is a polynomial expression in $n$ that has degree $m-1$ when $i =0$ and degree less than or equal to $ m-1$ for $1 \leq i \leq k-2$. These facts together with Lemma \ref{ascenttok} imply
	\begin{align*}
		\lvert \ATK \rvert&= \sum_{i=0}^{k-2} \binom{k-1}{i} \binom{n-k}{m-i-1} \lvert P(S_1;m-1)\rvert  2^{k-i-2} -\lvert P(S_1;n)^{{\nearrow k}}\rvert-\lvert P(S_2;n)^{{\nearrow k}}\rvert.\\
				&= \sum_{i=0}^{k-2} \binom{k-1}{i} \binom{n-k}{m-i-1} \lvert P(S_1;m-1)\rvert 2^{k-i-2} - p_{\a_1(k)}(n)-p_{\a_2(k)}(n)\\
				&= p_{\a(k)}(n)
	\end{align*}
 is a polynomial in the variable $n$ of degree $m-1$ that returns integer values when evaluated at integers.
\end{proof}

Below we show an example for how to find the polynomial $p_{\a(k)}(n)$.
\begin{example}
It is well known that any sequence given by a polynomial of degree $d$ can be completely determined by any consecutive $d+1$
values
by the method of finite differences
\cite[Proposition 1.9.2]{S12}.
Theorems \ref{Pv1} and \ref{Pv2} give us a way of finding explicit formulas $p_{\a(k)}(n)$ and $p_\d(n)$ for an admissible set
$S$.

For instance if $S=\{2,4\}$ and $k =6$,  Theorem \ref{Pv2} tells us $p_{\a(k)}(n)$ is a polynomial of degree three. Hence we
require four consecutive terms to compute $p_{\a(k)}(n)$.
One can compute that the first few values of $ p_{\a(6)}(n) = |P(S;n)^{{\nearrow 6}}|$ are as follows:
\[ p_{\a(6)}(6) =  16, \  p_{\a(6)}(7)=  80,  \ p_{\a(6)}(8) =  240, \ p_{\a(6)}(9) = 480, \ p_{\a(6)}(10)= 880, \ p_{\a(6)}(11)=
1456, \  \ldots  \]
We then take four successive differences until we get a row of zeros in the following array:
 \[\begin{matrix} 16 & 80 & 224 & 480 & 880 & 1456 & \cdots \\
                  64 & 144 & 256 & 400 & 576 & \cdots  & \\
                  80 & 112 & 144  & 176 & \cdots   &  & \\
                  32 & 32 & 32  &  \cdots  &    &  & \\
                  0 & 0 &  \cdots &      &    &  &  \end{matrix} \]
The polynomial $p_{\a(6)}(n)$ can be expressed in the basis $\binom{n} {j}$ as \[ p_{\a(6)}(n) = 16 \binom{n}{6} + 64
\binom{n}{7}
+ 80 \binom{n}{8} + 32 \binom{n}{9} . \]
We note that the sequence given by
$p_{\a(6)}(n)/16$ in this example is sequence \textcolor{blue}{\href{http://oeis.org/A000330}{A000330}} in \emph{Sloane's Online
Encyclopedia of Integer Sequences (OEIS)} with the index
$n$ shifted by $6$ \cite[A000330]{OEIS}.
\end{example}

\section{Pattern bundles of Coxeter groups of type \texorpdfstring{$C$}{C} and \texorpdfstring{$D$}{D}}\label{partitionbd}
In this section, we describe embeddings of the Coxeter groups of type $C_n$ and $D_n$ into the symmetric group $\S_{2n}$.
We then partition these groups into subsets, which we call \emph{pattern bundles} and denote by $\B_n(\pi)$ and $\D_n(\pi)$, that
correspond to permutations $\pi$ of $\S_n$.
Each of the \emph{type $C_n$ pattern bundles}  $\B_n(\pi)$  contain $2^n$ elements, and the \emph{type $D_n$ pattern bundles}
$\D_n(\pi)$ contain  $2^{n-1}$ elements.
These sets allow us to give a concise proof of Theorem \ref{HIPB},
and they play an instrumental role in our proof of Theorem \ref{key}.
\subsection{Pattern Bundle Algorithms for $\B_n$ and $\D_n$}
We define the group of type $C_n$ \emph{mirrored permutations} to be the subgroup
$\B_n \subset \S_{2n}$ consisting of all permutations $\pi_1\pi_2\cdots\pi_n \vert  \pi_{n+1}\pi_{n+2}\cdots\pi_{2n} \in \S_{2n}$
where $\pi_i = k$  if and only if $\pi_{2n-i+1}=2n-k+1$.
In other words, we place a ``mirror'' between $\pi_n$ and $\pi_{n+1}$, then the numbers $i$ and $2n-i+1$ must be the same
distance
from the mirror for each $1 \leq i \leq n$.
A simple transposition $s_i$ with $1 \leq i \leq n-1$ acts on a mirrored permutation $\pi \in \B_n \subset \S_{2n}$ (on the
right)
by simultaneously transposing $\pi_i$ with $\pi_{i+1}$ and $\pi_{2n-i}$ with $\pi_{2n-i+1}$.
The simple transposition $s_n$ acts on a mirrored permutation $\pi \in \B_n \subset \S_{2n}$ by transposing $\pi_n$ with
$\pi_{n+1}$.

Similarly, we define the group of type $D_n$ \emph{mirrored permutations} as the subgroup
$\D_n \subset \S_{2n}$ consisting of all permutations $\pi_1\pi_2\cdots\pi_n \vert  \pi_{n+1}\pi_{n+2}\cdots\pi_{2n} \in \S_{2n}$
where $\pi_i = k$  if and only if $\pi_{2n-i+1}=2n-k+1$ and the set $\{ \pi_1,\pi_2, \cdots, \pi_n \}$ always contains an even
number of elements from the set $\{ n+1, n+2, \ldots, 2n\}$.
A simple transposition $s_i$ with $1 \leq i \leq n-1$ acts on a mirrored permutation $\pi \in \D_n \subset \S_{2n}$ (on the
right)
by simultaneously transposing $\pi_i$ with $\pi_{i+1}$ and $\pi_{2n-i}$ with $\pi_{2n-i+1}$.
The simple transposition $s_n$ acts on a mirrored permutation $\pi \in \D_n \subset S_{2n}$ by transposing $\pi_{n-1}\pi_n$ with
$\pi_{n+1}\pi_{n+2}$.

\begin{definition}\label{relorder} Let $\pi=\pi_1\pi_2\cdots\pi_n\in \S_n$.
We define the \emph{pattern bundle of $\pi$} in type $C_n$ (denoted $\B_n(\pi)$) to be the set of all mirrored permutations
\[\tau=\tau_1\tau_2\cdots\tau_n\vert \tau_{n+1}\tau_{n+2}\cdots\tau_{2n}\in \B_n\] such that $\tau_1\tau_2\cdots\tau_n$ has the
same relative order as $\pi_1\pi_2\cdots\pi_n$.
\end{definition}

We could equivalently describe $\B_n(\pi)$ as the set of mirrored permutations which contain the \emph{permutation pattern} $\pi$
in the first $n$ entries.
 We will show that these sets partition $\B_n$ into subsets of size $2^n$.
For every $\pi\in\S_n$, we will describe how to construct the \emph{pattern bundle} $\B_n(\pi)\subset \B_n$ of $\pi$ using the
following process:

\noindent \begin{algorithm}[Pattern Bundle Algorithm for $\B_n(\pi)$]\label{algorithmB} \
\begin{enumerate}
\item Let $\pi=\pi_1\pi_2\cdots\pi_n \in\S_n$ and write it as a mirrored permutation \[\pi_1\pi_2\cdots\pi_n\vert
\pi_{n+1}\pi_{n+2}\cdots\pi_{2n}\in \S_{2n}.\]
\item Let $I_n=\{\pi_1,\pi_2,\ldots,\pi_n\}$. Fix $0\leq j\leq n$ and select $j$ elements from the set $I_n$. Then let $\Pi$ be
the set consisting of the $j$ selected elements. \label{fix}
\item The set $I_n \setminus \Pi$ consists of $n-j$ elements. Denote this subset of $I_n$ by $\Pi^c$.

\item Let $\overline{\Pi^c}$ denote the set containing $\pi_{2n-i_k+1} = 2n - \pi_{i_{k}}+1$, for each $\pi_{i_k} \in \Pi^c$.
Note
that $|\overline{\Pi^c}|=n-j$.\label{switch}
\item List the $n$ elements of the set \[ \overline{\Pi^c}
\sqcup \Pi \] so that they are in the same relative order as $\pi$ and call them $\tau_1 \tau_2 \cdots
\tau_n$.\label{samerelorder}
(Note that the set $\overline{\Pi^c}$ consists of the integers that were switched in Step \ref{switch}, and the set $\Pi$
consists
of the ones that were fixed in Step \ref{fix}.)
 \item The order of the remaining entries $\tau_{n+1} \tau_{n+2} \cdots \tau_{2n}$ is determined by that of $\tau_1 \tau_2 \cdots
\tau_n$ since we must have $\tau_{2n-i+1} = 2n - \tau_{i}+1$ for $1 \leq i \leq n$.
\item {Output} the mirrored permutation
$\tau_1\tau_2\cdots\tau_n\vert\tau_{n+1}\tau_{n+2}\cdots\tau_{2n}\in\B_n\subset\S_{2n}$ {and stop}.
\end{enumerate}
\end{algorithm}

Step \ref{samerelorder} ensures all of the constructed elements will have the same relative order as $\pi$.  It follows that the
set $\B_n(\pi)$ described in
Definition \ref{relorder} denotes all elements of $\B_n$ created from $\pi$ by Algorithm \ref{algorithmB}.
Notice in Step \ref{fix}, we must choose $j$ values to fix.
When we let $j$ range from $0$ to $n$ we see that
the total number of elements in $\B_n(\pi)$ is given by
\[\binom{n}{0}+\binom{n}{1}+\cdots+\binom{n}{n-1}+\binom{n}{n}=2^n.\]
We conclude that $|  \B_n(\pi)|=2^n$ for all $\pi \in \S_n$.

Note that if $\tau = \tau_1 \tau_2 \cdots \tau_n \vert  \tau_{n+1} \tau_{n+2} \cdots \tau_{2n} \in \B_n$, then $\tau_1 \cdots
\tau_n$
has same relative order as exactly one element $\pi \in \S_n$.
It follows that if $\sigma$ and $\pi$ are distinct permutations of $\S_n$, then $\B_n(\sigma )\cap\B_n(\pi)=\emptyset$.
Therefore, this process creates all $2^n n!$ elements of $\B_n$.

\begin{example}\label{B_3}
Using Algorithm \ref{algorithmB}, we have partitioned the elements of $\B_3$ into the pattern bundles $\B_n(\pi)$.
 \[\B_3  =  \left \{
 \begin{matrix}  \mathbf{123\vert 456}  \\ 124\vert 356 \\ 135 \vert  246 \\ 145\vert  236   \\ 236\vert  145 \\ 246 \vert  135
\\
356\vert 124  \\  456\vert  123   \end{matrix} \ \
\begin{matrix} \mathbf{ 132\vert 546}  \\ 142\vert 536 \\ 153 \vert  426  \\ 154\vert  326 \\ 263\vert  415  \\ 264 \vert  315 \\
365\vert 214  \\  465\vert  213   \end{matrix}  \ \
 \begin{matrix}  \mathbf{213\vert 465}  \\ 214\vert 365  \\ 315 \vert  264 \\ 326\vert  154 \\ 415\vert  326  \\ 426 \vert  153
\\
536\vert 142  \\  546\vert  132  \end{matrix}  \ \
\begin{matrix} \mathbf{ 231\vert 645 } \\ 241\vert 635 \\ 351 \vert  624 \\ 362\vert  514 \\ 451\vert  623  \\ 462 \vert  513 \\
563\vert 412  \\ 564\vert  312   \end{matrix}  \ \
\begin{matrix}  \mathbf{312\vert 564 } \\ 412\vert 563 \\ 513 \vert  462  \\ 514\vert  362 \\ 623\vert  451 \\ 624 \vert  351  \\
635\vert 241  \\ 645\vert  231   \end{matrix}   \ \
\begin{matrix}  \mathbf{321\vert 654}  \\ 421\vert 653 \\ 531 \vert  642  \\ 541\vert  632  \\ 632\vert  541\\ 642 \vert  531 \\
653\vert 421  \\  654\vert  321   \end{matrix} \ \
 \right  \}  \]
 One can see that
the elements of $\pi \in \S_3$ correspond to the elements in the top row (in bold font).
Each column consists of the pattern bundle $\B(\pi)$ corresponding to each $\pi\in\S_3$.

\end{example}

\begin{definition}\label{relorderD} Let $\pi=\pi_1\pi_2\cdots\pi_n\in \S_n$. We define the \emph{pattern bundle} $\D_n(\pi)$ to
be
the set of
all mirrored permutations $\tau=\tau_1\tau_2\cdots\tau_n\vert \tau_{n+1}\cdots\tau_{2n}\in \D_n$ such that
$\tau_1\tau_2\cdots\tau_n$ has the same
relative order as $\pi_1\pi_2\cdots\pi_n$.
\end{definition}

For every $\pi\in\S_n$, we construct the subsets $\D_n(\pi)\subset\D_n$ using the following process:

\noindent \begin{algorithm}[Pattern Bundle Algorithm for $\D_n(\pi)$]\label{algorithmD} \
\begin{enumerate}
\item Let $\pi=\pi_1\pi_2\cdots\pi_n \in\S_n$ and write it as a mirrored permutation \[ \pi_1\pi_2\cdots\pi_n\vert
\pi_{n+1}\pi_{n+2}\cdots\pi_{2n}\in \S_{2n}. \]
\item If $n$ is even then proceed to Step \ref{fixeven}, if it is odd proceed to Step \ref{fixodd}.

\noindent\textbf{Even}
\item If $n$ is even, then pick an even number $2j$, with $0\leq j\leq \frac{n}{2} $. Select a subset of $2j$ elements from the
set
$\{\pi_1,\pi_2,\ldots,\pi_n\}$ to keep fixed. Then let $\Pi$ be the set consisting of the $2j$ selected elements. \label{fixeven}
\item Let the set of remaining $n-2j$ elements be $\Pi^c$. (Note that $n-2j$ is an even integer.)
\item Let the set $\overline{\Pi^c}$ denote the set of mirror images from the elements of $\Pi^c$.
In other words, for each $\pi_{i_k}\in \Pi^c$ the mirror image $\pi_{2n-i_k+1} \in \overline{ \Pi^c}$.
\item List elements of the set $\Pi \sqcup \overline{\Pi^c}$
so they are in the same relative order as $\pi$ and call the resulting permutation $\tau_1 \tau_2 \cdots
\tau_n$.\label{samerelorderD}
\item Then the entries of $\tau_{n+1}\tau_{n+2} \cdots \tau_{2n}$ are determined by the relation $\tau_{2n-i_k+1} = 2n-
\tau_{i_k}+1$.
\item {Output} the mirrored permutation
$\tau_1\tau_2\cdots\tau_n\vert\tau_{n+1}\tau_{n+2}\cdots\tau_{2n}\in\D_n\subset\S_{2n}$ {and stop}.

\noindent\textbf{Odd}
\item If $n$ is odd, then pick an odd number $2j+1$ with $1\leq j\leq \frac{n-1}{2}$.   Select a subset of $2j+1$ elements from
the set $\{\pi_1,\pi_2,\ldots,\pi_n\}$
to keep fixed.
Call the set of these $2j+1$ elements $\Pi$. \label{fixodd}
\item Let the set of remaining $n-2j-1$ elements be denoted as $\Pi^c$. (Note that $n-2j-1$ is an even integer.)
\item Let the set $\overline{\Pi^c}$ denote the set of mirror images from the elements of $\Pi^c$.
In other words, for each $\pi_{i_k}\in \Pi^c$ the mirror image $\pi_{2n-i_k+1} \in \overline{ \Pi^c}$.
\item List elements of the set $\Pi \sqcup \overline{\Pi^c}$
so they are in the same relative order as $\pi$ and call the resulting permutation $\tau_1 \tau_2 \cdots \tau_n$. \label{step12}
\item Then the entries of $\tau_{n+1}\tau_{n+2} \cdots \tau_{2n}$ are determined by the relation $\tau_{2n-i_k+1} = 2n-
\tau_{i_k}+1$.
\item {Output} the mirrored permutation
$\tau_1\tau_2\cdots\tau_n\vert\tau_{n+1}\tau_{n+2}\cdots\tau_{2n}\in\D_n\subset\S_{2n}$ {and stop}.
\end{enumerate}
\end{algorithm}

By Definition \ref{relorderD} the set $\D_n(\pi)$ is the subset of all elements of $\D_n$ which are created from $\pi$ by
Algorithm \ref{algorithmD}.
This is because Steps \ref{samerelorderD} and \ref{step12} ensure that all of the constructed elements will have the same
relative
order as $\pi$.

In Steps \ref{fixeven} and \ref{fixodd} we choose an even/odd number of entries to fix, so that we always exchange an even number
of entries with their mirror image. This ensures each
$\tau$ constructed via Algorithm \ref{algorithmD} is a type $\D_n$ mirrored permutation.
When $n$ is even we can see from Step \ref{fixeven} that the total number of permutations created by Algorithm \ref{algorithmD}
is
given by $\sum_{j =0 }^{\frac{n}{2}}\binom{n}{2j}$. When $n$ is odd we can use the identity
\[\displaystyle\sum_{j =0 }^{\lfloor\frac{n}{2}\rfloor}\binom{n}{2j+1} = \sum_{k=0}^{\lfloor\frac{n}{2}\rfloor} \binom{n}{2k}
\text{ where } 2k=n-(2j+1) \]
 to see that the total number of elements created by Algorithm \ref{algorithmD} is also given by the formula
\[\displaystyle\sum_{j =0 }^{\lfloor\frac{n}{2}\rfloor}\binom{n}{2j}.\]

Pascal's identity for computing binomial coefficients states that for all integers $n$ and $k$ with $1\leq k \leq n-1$
\[\binom{n}{k}=\binom{n-1}{k-1}+\binom{n-1}{k} .\]
Using this identity we can see that
\[\displaystyle\sum_{j =0 }^{\lfloor\frac{n}{2}\rfloor}\binom{n}{2j}=\displaystyle\sum_{j=0}^{n-1}\binom{n-1}{j}=2^{n-1}.\]
So for every element $\pi\in\S_n$ we create $2^{n-1}$ elements of $\D_n$. Hence $\lvert  \D_n(\pi)\rvert=2^{n-1}$.
Also notice that for each choice of $\pi$ the $2^{n-1}$ elements of $\D_n(\pi)$ will be distinct due to the choice of which
elements get sent to their mirror image.
Namely, if $\sigma$ and $\pi$ are distinct permutations of $\S_n$, then $\D_n(\sigma)\cap\D_n(\pi)=\emptyset$.
Therefore, this process creates all $2^{n-1}n!$ distinct elements of $\D_n$.

\begin{example}Using Algorithm \ref{algorithmD}, we have partitioned the set $\D_3$ into the pattern bundles $\D_n(\pi)$.

 \[\D_3  =  \left \{\begin{matrix}  \mathbf{123\vert 456}  \\  145\vert  236  \\ 246 \vert  135 \\ 356\vert 124   \end{matrix} \ \
\begin{matrix} \mathbf{ 132\vert 546}  \\  154\vert  326  \\ 264 \vert  315 \\ 365\vert 214     \end{matrix}  \ \
 \begin{matrix}  \mathbf{213\vert 465}  \\ 214\vert 365  \\ 426 \vert  153 \\ 536\vert 142  \end{matrix}  \ \
\begin{matrix} \mathbf{ 231\vert 645 } \\ 241\vert 635 \\ 462 \vert  513 \\ 563\vert 412   \end{matrix}  \ \
\begin{matrix}  \mathbf{312\vert 564 } \\ 412\vert 563 \\  624 \vert  351  \\ 635\vert 241  \end{matrix}   \ \
\begin{matrix}  \mathbf{321\vert 654}  \\ 421\vert 653  \\ 642 \vert  531 \\ 653\vert 421  \end{matrix} \ \
 \right  \} . \]
Note that the elements in the top row (in bold font) are the elements of $\S_3$,
while the elements in each column are the elements of the pattern bundle $\D_3(\pi)$, for each respective $\pi\in\S_3.$
\end{example}

\subsection{Peak sets in types \texorpdfstring{$B$}{B} and \texorpdfstring{$C$}{C}} \label{comparison}
Castro-Velez et al. \cite{CV14} studied the sets of type $\BS_n$ \emph{signed permutations} (defined below) with a given peak set
$\STT \subset [n-1]$.
It is well-known that the group of signed permutations of type $\BS_n$ is isomorphic to the Coxeter groups of type $B_n$ and
$C_n$.
In this section, we describe one such isomorphism between the group of signed permutations $\BS_n$ and
the mirrored permutations $\B_n$ and {show how} the peak sets in mirrored permutations studied in this paper
correspond with the ones
studied by Castrol-Velez et al. \cite{CV14}.
It is important to know that even though we compute the cardinalities of similar sets, our methods are completely different and
yield different equations.
In particular, Castro-Velez et al. use induction arguments similar to those used by Billey, Burdzy, and Sagan in the realm of
signed permutations to derive their formulas, whereas we
use the pattern bundles of type $\B_n$ to {reduce} the problem to calculations in the symmetric group.

Let $\BS_n$ denote the group of \emph{signed permutations} on $n$ letters
\[\BS_n := \left \{\b_1\b_2\cdots \b_n \, \mid \, \b_i \in \{ -n,-n+1, \ldots,-1,1, \ldots, n \} \text{ and } \{|\b_1|,|\b_2|,
\ldots,
|\b_n|\} = [n] \right  \}  .\]
We say that a signed permutation $\beta \in \BS_n$ has a peak at $i$ if $\b_{i-1}< \b_i > \b_{i+1}$.

\begin{definition}
 Let $\ST \subseteq [n-1]$ then the sets $P_B(\ST;n)$ and $\hat{P}_B(\STT;n)$ are defined to be
\begin{align*}
P_B(\ST;n) &:=\{\b \in \BS_n \, \mid \, P(\b_1\cdots\b_n)=\ST \}  \\
\hat{P}_B(\STT;n) &:=\{\b \in \BS_n \, \mid \, P(\b_0\b_1\cdots\b_n)=\STT, \text{ where } \b_0=0 \}.
\end{align*}
\end{definition}
\noindent
In this paper we study the sets of mirrored permutations of type $C_n$ and $D_n$ that have a given peak set $S$.
\begin{definition}\label{defpeakset} Let $\B_n$ and $\D_n$ be the mirrored permutations of types $C$ and $D$, respectively.
For $S \subseteq [n-1]$ we define the sets $P_C(S;n)$ and $P_D(S;n)$ as
\begin{align}
P_C(S;n) &:=\{\pi\in \B_n \, \mid \, P(\pi_1\pi_2\cdots\pi_n)=S\},\\
P_D(S;n)&:=\{\pi\in\D_n \, \mid \, P(\pi_1\pi_2 \cdots  \pi_n)=S\}.
\end{align}
Let $S \subseteq [n]$ we define the sets $\hat{P}_C(S;n)$ and $\hat{P}_D(S;n)$ as
\begin{align}
\hat{P}_C(S;n)&:=\{\pi\in\B_n \, \mid \, P(\pi_1\pi_2\cdots \pi_n|\pi_{n+1})=S\},\\
\hat{P}_D(S;n)&:=\{\pi\in\D_n \, \mid \, P(\pi_1\pi_2\cdots\pi_n|\pi_{n+1})=S\}.
\end{align}
\end{definition}
Note that $\hat{P}_C(S;n)$ and $ \hat{P}_D(S;n) $ differ from $P_C(S;n)$ and $P_D(S;n)$ in that they allow for a peak in the
$n\text{th}$ position when $\pi_{n-1} < \pi_n > \pi_{n+1}$.
{
The following proposition provides a bijection between the peak sets $\hat{P}_B(\STT;n)$ considered by Castro-Velez et al. \cite{CV14} and
$\hat{P}_C(S;n)$ considered in this paper.
\begin{proposition}\label{prop:compare1}
Let $S = \{ i_1, i_2, \ldots, i_k \}\subset \{2,3,\ldots, n\}$ and \[ \STT =\{ n-i_1+1, n-i_2+1, \ldots, n-i_k+1 \}\subset [n-1],
\] then  there is a bijection between $\B_n$ and $ \BS_n$ that maps $\hat{P}_C(S;n)$ to $\hat{P}_B(\STT;n)$.
\end{proposition}}

{The above result states that} the peaks of $\pi_1 \pi_2\cdots \pi_n|\pi_{n+1}$ correspond bijectively with the peaks
of a signed permutation $\beta_0\beta_1\beta_2 \cdots \beta_n$
where $\beta_0=0$, and the peaks of $\pi_1\pi_2\cdots \pi_n$
correspond with those of $\beta_1\beta_2 \cdots \beta_n$.
{Before proceeding to the proof of Proposition \ref{prop:compare1} we set some preliminaries.}

Billey and Lakshmibai note that a mirrored permutation
\[ \pi_1\pi_2\cdots\pi_n|\pi_{n+1}\pi_{n+2}\cdots \pi_{2n} \in \B_n \] can be represented by either sides of the mirror
$\pi_1\pi_2\cdots \pi_n$
or $\pi_{n+1}\pi_{n+2} \cdots \pi_{2n}$ \cite[Definition 8.3.2]{BL00}, and
we use the latter $\pi_{n+1}\pi_{n+2} \cdots \pi_{2n}$ to define a map $\F$ from $\B_n$ to $\BS_n$ as follows:
\begin{align*}
\F :  \B_n & \rightarrow \BS_n \\
                      \pi_1 \pi_2 \cdots \pi_n|\pi_{n+1}\pi_{n+2} \cdots \pi_{2n}&  \mapsto \b_1\b_2 \cdots \b_n
\end{align*} where \[ \b_i = \begin{cases} \pi_{n+i}-n & \text{ if } \pi_{n+i} > n \\ \pi_{n+i} -n-1 & \text{ if } \pi_{n+i} \leq
n. \end{cases}\]

{We} consider a signed permutation $\b=\b_1\b_2 \cdots \b_n$ in $\BS_n$ as $\b_0\b_1\cdots \b_n$ where $\b_0=0$,
thus allowing for a peak at position 1.
We note that the map $\F$ respects the relative order of the sequence $\pi_n\pi_{n+1}\pi_{n+2} \cdots \pi_{2n}$ i.e., for $0\leq
i\leq n-1$ if $\pi_{n+i}<\pi_{n+i+1}$ then $\b_i<\b_{i+1}$,
and similarly if $\pi_{n+i}>\pi_{n+i+1}$ then $\b_i>\b_{i+1}.$

We also define an automorphism $\G: \BS_n \rightarrow \BS_n$ which switches the sign of each $\b_i$ in $\b_0\b_1\b_2\cdots \b_n$
(keeping $\b_0=0$ fixed). To avoid cumbersome notation, for each
$\b_i$ we set $\overline{\b_i}=-\b_i$.
The following table illustrates how the maps $\F$ and $\G$ map the group of mirrored permutations $\B_2$ bijectively to the group
of signed permutations $\BS_2$.

\begin{center}
\begin{tabular}{|c|c|c|}  \hline
  $\pi \in \B_2$ &  $\F(\pi) \in \BS_2$ & $\G(\F(\pi)) \in \BS_2$ \\ \hline
   & &\\
  $12\vert 34$ & $012$ & $0\overline{1}\overline{2}$ \\
  $21\vert 43$ & $021$ & $0\overline{2}\overline{1}$ \\
  $13\vert 24$ & $0\overline{1} 2$ & $01 \overline{2}$\\
  $24\vert 13$  & $0\overline{2} 1$ & $02\overline{1} $    \\
  $31\vert 42$  & $0 2 \overline{1}$ & $0\overline{2}1$     \\
  $42\vert 31$  & $01\overline{2}$ & $0\overline{1}2$    \\
   $34\vert 12$  &$0\overline{21} $ & $021$  \\
   $43\vert 21$  & $0\overline{12}$ & $012$    \\
    && \\ \hline
\end{tabular}
\end{center}

{With the above notation at hand we now proceed to the proof.}

\begin{proof}[Proof of Proposition \ref{prop:compare1}]
 Let $\pi=\pi_1\pi_2\cdots \pi_n| \pi_{n+1} \cdots \pi_{2n}$ be a mirrored permutation and $\F(\pi)=
\b_0\b_1\b_2\cdots \b_n$. Then we see $\G(\F(\pi))=\b_0\overline{\b_1} \overline{\b_2}\cdots \overline{\b_n}$.
 Suppose $\pi_i<\pi_{i+1}$ for some $i\in \{1,2,\ldots, n\}.$ Looking at the mirror images of $\pi_i$ and $\pi_{i+1}$ we get
$2n-\pi_i+1>2n-\pi_{i+1}+1$, thus $\pi_{2n-i+1}>\pi_{2n-(i+1)+1}$.
 Since the map $\F$ respects the relative order of $\pi_n \pi_{n+1} \cdots \pi_{2n}$ we have $ \b_{n-i+1}>\b_{n-(i+1)+1},$ and
thus $
 \overline{\b_{n-i+1}}<\overline{\b_{n-(i+1)+1}}.$
 Using the same argument but replacing ``$<$'' with ``$>$'' and vice-versa we get that if $\pi_i>\pi_{i+1}$ then
$\overline{\b_{n-i+1}}>\overline{\b_{n-(i+1)+1}} .$
 Therefore if $\pi \in \hat{P}_C(S;n)$ then $\G(\F(\pi))\in \hat{P}_B(\STT;n)$,
 and if $\pi \not\in \hat{P}_C(S;n)$ then $\G(\F(\pi))\not\in \hat{P}_B(\STT;n)$.
 Since both $\G$ and $\F$ are bijections we conclude that $\G(\F( \hat{P}_C(S;n)))=\hat{P}_B(\STT;n).$
 \end{proof}

We can also consider signed permutations $\b \in \BS_n$ without the convention that $\b_0=0$. In that case we obtain the
following
result.

\begin{corollary}
Let $S = \{ i_1, i_2, \ldots, i_k \}\subset \{2,3,\ldots, n-1\}$ and \[ \ST =\{ n-i_1+1, n-i_2+1, \ldots, n-i_k+1
\}\subset\{2,\ldots, n-1\},\]
then the bijection $\G \circ \F : \B_n \rightarrow \BS_n$ maps $P_C(S;n)$ to $P_B(\ST;n)$.
\end{corollary}
\begin{proof}
The proof of this corollary proceeds exactly as the proof of Proposition \ref{prop:compare1}.
\end{proof}

\subsection{The sets \texorpdfstring{$P_C(S;n)$}{PC(S;n)} and \texorpdfstring{$P_D(S;n)$}{PD(S;n)}}
In this subsection, we use the fact that $\B_n(\pi)$ and $\D_n(\pi)$ partition $\B_n$ and $\D_n$ to give concise proofs that
$|P_C(S;n)|=p(n)2^{2n-|S|-1}$
and $| P_D(S;n)|  = p(n) 2^{2n-| S|-2}$ where $p(n)$ is the polynomial given in Theorem 2.2 of Billey, Burdzy, and Sagan's paper
\cite[Theorem 2.2]{BBS13}.

\begin{theorem}\label{HIPB}
Let $S \subseteq [n-1]$, then
	\begin{enumerate}[(I)]
		\item \label{HIPB1}$|P_C(S;n)|=p(n)2^{2n-|S|-1}$
		\item \label{HIPB2}$|P_D(S;n)|=p(n)2^{2n-|S|-2}$.
	\end{enumerate}

\end{theorem}

\begin{proof}
 To prove part ($\ref{HIPB1}$) note that Billey, Burdzy, and Sagan showed that $|P(S;n)| = p(n) 2^{n-|S|-1}$ where $p(n)$ is a polynomial with degree $\max(S)-1$,
\cite[Theorem 2.2]{BBS13}.
 Algorithm \ref{algorithmB} showed that each $\pi \in P(S;n)$ corresponds to a subset $\B_n(\pi) \subset \B_n$ which contains
$2^{n}$ elements.
 By construction these elements have the exact same peak set as $\pi$. In other words, for every $\tau\in \B_n(\pi)$ the peak
sets
$P(\tau)=P(\pi)=S$ agree.
 We compute that $ | P_C(S;n) | = p(n) 2^{n-|S|-1} 2^n = p(n) 2^{2n-|S|-1}$.

 Part (\ref{HIPB2}) follows similarly, replacing  $\B_n$ with $\D_n$, Algorithm \ref{algorithmB} with Algorithm \ref{algorithmD}, and $2^n$ with $2^{n-1}$.
 \end{proof}

%
%
%
%
%
%
%
\section{Peak sets of the Coxeter groups of type \texorpdfstring{$C$}{C} and \texorpdfstring{$D$}{D}} \label{typeCD}
{
In this section we use specific sums of binomial coefficients and the partitions \[P(S;n) = \cielo
\sqcup \infierno\mbox{ and }\cielo = \sqcup_{k =2}^n \ATK\] to describe
the cardinality
of the sets $\hat{P}_C(S;n)$, $\hat{P}_C( S \cup \{ n \}; n)$, $  \hat{P}_D(S;n)$, and $
\hat{P}_D(S\cup\{n\};n)$.
We begin by setting the following notation:}

 \begin{definition} \label{phipsi} Let $\Phi(n,k)$ denote the sum of the last $n-j+1 $ terms of the $n^{\text{th}}$ row in
Pascal's triangle:
 \[ \Phi(n,k) = \sum_{i=k}^n \binom{n}{i} = \binom{n} {k} + \binom{n}{k+1} +\cdots + \binom{n}{n} ,\]
 and let \[\Psi(n,k)=2^n-\Phi(n,k).\]
 \end{definition}

 {We can now state our main result.}

 \begin{theorem}\label{key}
 \underline{Type $C$}: Let $\hat{P}_C(S;n)$ denote the set of elements of $\B_n$ with peak set $S\subset [n-1]$, then
 \[ \lvert  \hat{P}_C(S;n) \rvert =  \displaystyle\sum_{k=1}^{n}\lvert  \ATK \rvert \cdot  \Phi(n,k) + \lvert
\infierno\rvert\cdot 2^n  \] \mbox{ and}
\[ \lvert   \hat{P}_C(S\cup\{n\};n) \rvert =\sum_{k=1}^n\lvert   \ATK \rvert\cdot\Psi(n,k). \]

\noindent
\underline{Type $D$}: Let $\hat{P}_D(S;n)$ denote the set of elements of $\D_n$ with peak set $S\subset  [n-1]$. If $n$ is even, then
\[ \lvert  \hat{P}_D(S;n)\rvert=\displaystyle\sum_{k=1}^{ \frac{n}{2} } \big(\lvert  P(S;n)^{{\nearrow
2k-1}}\rvert+\lvert
P(S;n)^{{\nearrow 2k}}\rvert\big)\Phi(n-1,2k-1)+\lvert  \infierno\rvert2^{n-1} \]   and
\[\lvert  \hat{P}_D(S\cup\{n\};n)\rvert=\displaystyle\sum_{k=1}^{ \frac{n}{2} }\big(\lvert  P(S;n)^{{\nearrow
2k-1}}\rvert+\lvert
P(S;n)^{{\nearrow 2k}}\rvert\big)\Psi(n-1,2k-1).\]
If $n$ is odd, then
\[ \lvert  \hat{P}_D(S;n)\rvert=\displaystyle\sum_{k=1}^{ \frac{n-1}{2}} \big(\lvert  P(S;n)^{{\nearrow
2k+1}}\rvert+\lvert
P(S;n)^{{\nearrow 2k}}\rvert\big)\Phi(n-1,2k)+\lvert  \infierno\rvert2^{n-1} \]   and
\[\lvert  \hat{P}_D(S\cup\{n\};n)\rvert=\displaystyle\sum_{k=1}^{\frac{n-1}{2}}\big(\lvert  P(S;n)^{{\nearrow
2k+1}}\rvert+\lvert
P(S;n)^{{\nearrow 2k}}\rvert\big)\Psi(n-1,2k).\]

\end{theorem}

{ Since the proofs of the type $C$ and $D$ results in Theorem \ref{key} require some specific identities involving the functions 
$\Phi$ and $\Psi$, we present these results and proofs in Subsections \ref{typeCpeakset} and \ref{typeDpeaksets}, respectively.}

{We note that Proposition \ref{prop:compare1} shows that $|\hat{P}_B(\ST;n) |=|\hat{P}_C(S;n)|$.
 Castro-Velez et al. gave a recursive formula for computing the cardinality of the set $\hat{P}_B(\ST;n)$ \cite[Theorem
3.2]{CV14}. Theorem \ref{key} provides an alternate formula for $|\hat{P}_C(S;n)|=|\hat{P}_B(\ST;n) |$ using
 the sums of binomial coefficients $\Phi(n,k)$ and $\Psi(n,k)$, and the cardinalities of sets
$\infierno$ and $\ATK$.
}

\subsection{Peak sets of the Coxeter groups of type \texorpdfstring{$C$}{C}} \label{typeCpeakset}

The following lemma uses the functions $\Phi(n,k)$ and $\Psi(n,k)$ to count the number of elements in $\B_n(\pi)$ having an
ascent
in the $n^{\text{th}}$ position.
This lemma is the key step in the {type $C$} proof of Theorem \ref{key}.

\begin{lemma} \label{Happy}
 If $\pi \in \ATK $ then there are $\Phi(n,k)$ elements $\tau \in \B_n(\pi)$ with $\tau_n \leq n$ and  $\Psi(n,k)$
elements $\tau \in \B_n(\pi)$ with $\tau_n > n$.
\end{lemma}

\begin{proof}
Suppose that $\pi = \pi_1 \pi_2 \cdots \pi_n \in \ATK$, so $\pi_{n-1}<\pi_n = k$.
If $\tau = \tau_1 \tau_2 \cdots \tau_n \vert  \tau_{n+1} \tau_{n+2} \cdots \tau_{2n} \in\B_n(\pi)$, then $\tau_n$  is the
$k^\text{th}$ largest integer in the set $\{ \tau_1, \tau_2, \ldots, \tau_n \}$ because $\tau$ has the same relative order as
$\pi$ and $\pi_n = k$.
Therefore if at least $k$ elements of the set  $\{ \tau_1, \tau_2, \ldots, \tau_n \}$ have $\tau_i \leq n$ then we conclude that
$\tau_n \leq n$.

We will show there are $\binom{n}{j}$ elements $\tau \in \B_n(\pi)$ where exactly $j$ elements of the set $\{ \tau_1, \tau_2,
\ldots, \tau_n \}$ satisfy $\tau_n \leq n$.
To construct  such a $\tau$ we start with $\pi = \pi_1 \pi_2 \cdots \pi_n$, and then we choose $j$ elements of the set $\{ \pi_1,
\pi_2, \ldots, \pi_n\}$ to remain fixed. We replace the remaining $n-j$ elements of
  $\{ \pi_1, \pi_2, \ldots, \pi_n\}$ with their mirror images, which are all greater than $n$.  Finally, we list the elements of
the resulting set so that they have the same relative order as $\pi$
and call them $\tau_1\tau_2 \cdots \tau_n$. The subpermutation $\tau_{n+1} \tau_{n+2} \cdots \tau_{2n}$ is then completely
determined by the subpermutation $\tau_1\tau_2\ldots \tau_n$.   Thus there are $\binom{n}{j}$ mirrored permutations $\tau$ of the
form
  $\tau = \tau_1 \tau_2 \cdots \tau_n \vert  \tau_{n+1} \tau_{n+2} \cdots \tau_{2n} \in \B_n(\pi)$ where $j$ of the elements in
$\{ \tau_1, \tau_2, \ldots, \tau_n \}$ satisfy $\tau_i \leq n$.

 Considering all integers $j$ with $k \leq j \leq n$ we see that the number of elements  \[ \tau = \tau_1 \tau_2 \cdots \tau_n
\vert  \tau_{n+1} \tau_{n+2} \cdots \tau_{2n} \in \B_n(\pi)\] with at least $k$ of the elements
 in $\{ \tau_1, \tau_2, \ldots, \tau_n \}$ satisfying $\tau_i \leq n$ is exactly
\[\Phi(n,k)=\displaystyle\sum_{j=k}^n\binom{n}{j}.\]
Thus there are $\Phi(n,k)$ elements $\tau \in \B_n(\pi)$
 with $\tau_n \leq n$.   The other $2^n - \Phi(n,k) = \Psi(n,k)$ elements of $\B_n(\pi)$ must have $\tau_n >n$.
\end{proof}

{With the above result at hand, we now give the following proof.}

\begin{proof}[Proof of Theorem \ref{key} type $C$]
Let $\pi=\pi_1\pi_2\cdots\pi_n\in\Pn$, and recall that $\B_n(\pi)$ is the set of elements of $\B_n$ whose first $n$ entries have the same relative order as $\pi$, and $\rvert\B_n(\pi)\lvert = 2^n$ for any $\pi \in \S_n$.
Let $\tau = \tau_1\tau_2 \cdots \tau_n\vert \tau_{n+1} \cdots \tau_{2n} \in \B_n$ denote a mirrored permutation of type $C_n$, then there are two possibilities:
 \begin{itemize}
 \item either $\tau$ has the same peak set as $\pi$ so that $\tau \in \hat{P}_C(S;n)$, or
 \item $\tau $ has an additional peak at $n$, in which case $\tau \in \hat{P}_C(S\cup\{n\};n)$.
 \end{itemize}

\noindent \textbf{There are two cases in which $\tau \in \hat{P}_C(S;n)$:} \\
\noindent \textbf{Case 1: }  If $\pi$ ends with a descent, i.e., $\pi_{n-1} > \pi_n$, then every $\tau \in \B_n(\pi)$ also has
$\tau_{n-1} >\tau_n$, and thus $\tau$ is in $\hat{P}_C(S;n)$ because it cannot possibly have a peak at $n$ if it ends with a
descent.
We conclude that if $\pi \in \infierno$ then all $2^n$ elements $\tau \in \B_n(\pi)$ are in
$\hat{P}_C(S;n)$.  \\
\noindent \textbf{Case 2:}  If  $\pi$ ends with an ascent, i.e., $\pi_{n-1}< \pi_n$, then $\tau_{n-1} < \tau_n$ for all $\tau \in
 \B_n(\pi)$ as well.
 (Recall that for any $\sigma\in \B_n$ our map into $\S_{2n}$ identifies $\sigma_{i}$ with its \emph{mirror} $\sigma_{n-i+1}$ by
$
\sigma_{n-i+1} = 2n-\sigma_i+1$.)
Hence, if $\tau_n \leq n$, then $\tau_{n+1} = 2n -\tau_n+1 > \tau_n$.
In this case $\tau_{n-1} < \tau_n < \tau_{n+1}$, and $\tau$ does not have a peak at $n$. So $\tau \in \hat{P}_C(S;n)$.
Therefore we conclude that if $\pi \in \cielo $ and if $\tau \in  \B_n(\pi)$ satisfies $\tau_n \leq
n$ then $\tau$ is an
element of $\hat{P}_C(S;n)$.
By Lemma \ref{Happy} we conclude that if $\pi \in \ATK $ then $\Phi(n,k)$ of the elements in $\B_n(\pi)$ are in
$\hat{P}_C(S;n)$.

\noindent \textbf{Case 3:} There is only one case in which $\tau \in \hat{P}_C(S \cup \{n \};n)$.
If $\pi \in \cielo$ and $\tau \in \B_n(\pi)$ is such that $\tau_n >n$, then $\tau$ must satisfy
$\tau_{n-1} <
\tau_n >
\tau_{n+1}$ because $\tau_{n+1} = 2n- \tau_n +1 <n$.
Therefore $\tau$ is an element of $\hat{P}_C(S\cup\{n\};n)$.  Applying Lemma \ref{Happy} we conclude that if $\pi \in P(S;n)^{{\nearrow k}}$ then $\Psi(n,k)$ of the elements in $\B_n(\pi)$ are in $\hat{P}_C(S \cup \{n \};n)$.


From Case 1 and Case 2 we conclude that the cardinality of $\hat{P}_C(S;n)$ is given by the formula
\[\lvert  \hat{P}_C(S;n)  \rvert=  \displaystyle\sum_{k=1}^{n}\lvert  \ATK\rvert  \cdot  \Phi(n,k) + \lvert
\infierno\rvert \cdot 2^n.\]
%
From Case 3 we get
\[\lvert   \hat{P}_C(S\cup\{n\};n) \rvert=\sum_{k=1}^n\lvert   \ATK\rvert \cdot \Psi(n,k).\]
\end{proof}
The following example illustrates the {type $C$} formulas proven in Theorem \ref{key}.
\begin{example}
Using the results of this section we compute the sets $\hat{P}_C(S;3)$, where $S\subset [3]$.
First we note that the group $\S_3$ can be partitioned as $\S_3=P( \emptyset;3 )\sqcup P(\{2\};3 )$, where \[P( \emptyset;3 ) =
\{
123, 321, 213, 312 \}\ \mbox{ and }\ P(\{2\};3 )= \{ 132, 231 \}.\]
To calculate the peak sets in type $C_3$ we will further partition the sets $P(\emptyset;3)$ and $P(\{2\};3 )$ using Definition~\ref{alpha-delta}. Hence
we compute
\[P(\emptyset;3) = {\underline{P(\emptyset;3)}} \sqcup P(\emptyset;3)^{{\nearrow 2}} \sqcup
P(\emptyset;3)^{{\nearrow 3}} ,\] where
${\underline{P(\emptyset;3)}} = \{ 321 \} $,
$P(\emptyset;3)^{{\nearrow 2}} = \{ 312 \}$,
and $P(\emptyset;3)^{{\nearrow 3}} = \{123, 213 \}$.

We also compute the set \[P(\{2\};3) = {\underline{P(\{2\};3)}} = \{231, 132 \}.\]
Of the 48 elements of the Coxeter group $\B_3$, only $2^3 \lvert  P(\emptyset;3)\rvert = 2^3 \cdot 4 = 32$ elements are in
$\hat{P}_C(\emptyset;3) \sqcup \hat{P}_C(\{3\};3)$.
Of these 32 permutations we observe that $18$ lie in $\hat{P}_C(\{3\};3)$ and $14$ lie in $\hat{P}_C(\emptyset;3)$.
We calculate $\lvert  \hat{P}_C(\emptyset;3)\rvert$ using Theorem \ref{key}:
\begin{align*}
 \lvert  \hat{P}_C(\emptyset;3)\rvert  = &  [  \lvert  {\underline{P(\emptyset;3)}} \rvert \cdot  2^3 ] +   [
\lvert
P(\emptyset;3)^{{\nearrow 2}}\rvert\cdot \Phi(3,2) ] +    [\lvert  P(\emptyset;3)^{{\nearrow
3}}\rvert  \cdot \Phi(3,3)]  \\
                                  = &    [1  \cdot 8  ]                            +   [1                       \cdot  4      ] +
[          2  \cdot 1   ]    = 14.
\end{align*}
Hence $\lvert  \hat{P}_C (\{3\};3) \rvert= 2^3 \cdot 4 - 14 = 18$.
Since $P(\{2 \};3) = {\underline{P(\{2 \};3)}}$ we have \[\lvert  \hat{P}_C( \{ 2 \};3) \rvert=  \lvert
{\underline{P(\{2 \};3)}}\rvert
\cdot  2^3   =  \lvert  P(\{2 \};3 )\rvert \cdot 2^3  =  16.\]
    Indeed one may confirm that  ${\hat{P}_C(\{2 \};3)}$ is the union of the following two sets:
\[\B_3 (231) =  \left \{\begin{matrix} & 231\vert 645 & \\ 241\vert 635 & 351 \vert  624 & 362\vert  514 \\ 451\vert  623  & 462
\vert  513 & 356\vert 124  \\ & 564\vert  312&   \end{matrix}  \right  \}  \text{ and }  \B_3 (132) =  \left \{\begin{matrix} &
132\vert 645 & \\ 142\vert 536 & 153 \vert  426 & 263\vert  415 \\ 154\vert  326  & 264 \vert  315 & 365\vert 214  \\ & 465\vert
213&   \end{matrix}  \right  \}  .\]
\end{example}

\subsection{Peak sets of the Coxeter group of type \texorpdfstring{$D$}{D}} \label{typeDpeaksets}
In this section, we use the functions $\Phi(n,k)$ and $\Psi(n,k)$ to describe the cardinalities of $ \hat{P}_D(S;n)$ and $
\hat{P}_D(S\cup\{n\};n)$.
The results depend on the parity of $n$. We begin by providing the following lemmas (similar to Lemma \ref{Happy}) which are used in the {type $D$} proof of Theorem \ref{key}.

\begin{lemma} \label{Happy1}
Let $n$ be even, and let $1 \leq k \leq n/2$.
 If $\pi \in P(S;n)^{\nearrow 2k}\sqcup P(S;n)^{\nearrow 2k-1}$, then there are $\Phi(n-1,2k-1)$ elements $\tau \in \D_n(\pi)$ with $\tau_n
\leq n$ and
 $\Psi(n-1,2k-1)$ elements $\tau \in \D_n(\pi)$ with $\tau_n > n$.
\end{lemma}

\begin{proof}
Suppose that $\pi = \pi_1 \pi_2 \cdots \pi_n \in P(S;n)^{{\nearrow 2k}}$, so $\pi_n =
2k$ and $\pi_{n-1} = i$ for some integer $i<2k.  $
If $\tau = \tau_1 \tau_2 \cdots \tau_n \vert  \tau_{n+1} \tau_{n+2} \cdots \tau_{2n} \in\D_n(\pi)$, then $\tau_n$  is the
$(2k)^\text{th}$ largest integer in the set
$\{ \tau_1, \tau_2, \ldots, \tau_n \}$ because $\tau$ has the same relative order as $\pi$ and $\pi_n = 2k$.
Therefore if at least $2k$ elements of the set  $\{ \tau_1, \tau_2, \ldots, \tau_n \}$ satisfy $\tau_i \leq n$ then we can
conclude that $\tau_n \leq n$.
Moreover, $\tau_n \leq n$ if and only if $\tau_n < \tau_{n+1} =2n-\tau_n+1$. Thus we wish to count the number of $\tau \in
\D_n(\pi)$ with $\tau_n \leq n$.

In the construction of $\D_n(\pi)$ the total number of $\tau$ with at least $2k$ of the elements from  $\{ \tau_1, \tau_2,
\ldots,
\tau_n \}$
fixed (and  less than or equal to $ n$) is given by the sum
	\begin{equation}\label{eq:nchoosek}
		 \binom{n}{2k}+ \binom{n}{2k+2} + \cdots + \binom{n}{n-2}+ \binom{n}{n},
	\end{equation} when $n$ is even.
Using the identity $\binom{n}{2k} = \binom{n-1}{2k-1} + \binom{n-1}{2k}$ we can see that the quantity in $(\ref{eq:nchoosek})$
equals
	\begin{align*}
		\left [\binom{n-1}{2k-1}+ \binom{n-1}{2k} \right
] +&\left [\binom{n-1}{2k+1}  + \binom{n-1}{2k+2} \right ]+ \cdots + \binom{n-1}{n-1} = \Phi(n-1,2k-1),   \end{align*}
 when $n$ is even.

Suppose that $\pi = \pi_1 \pi_2 \cdots \pi_n \in P(S;n)^{{\nearrow 2k-1}}$, so  $\pi_n
= 2k-1$ and $\pi_{n-1} = i$ for some integer $i<2k-1.  $
If $\tau = \tau_1 \tau_2 \cdots \tau_n \vert  \tau_{n+1} \tau_{n+2} \cdots \tau_{2n} \in\D_n(\pi)$, then $\tau_n$  is the
$(2k-1)^\text{th}$ largest integer in the set
$\{ \tau_1, \tau_2, \ldots, \tau_n \}$ because $\tau$ has the same relative order as $\pi$ and $\pi_n = 2k-1$.
Therefore if at least $2k-1$ elements of the set  $\{ \tau_1, \tau_2, \ldots, \tau_n \}$ satisfy $\tau_i \leq n$ then we can
conclude that $\tau_n \leq n$. Moreover, $\tau_n \leq n$ if and only if $\tau_n < \tau_{n+1} =2n-\tau_n+1$.
So again, the number of elements with $\tau_n \leq n$ is $\Phi(n-1,2k-1)$.

We conclude that when $\pi  \in P(S;n)^{{\nearrow 2k}} \sqcup P(S;n)^{{\nearrow 2k-1}}$ there are
$\Phi(n-1,2k-1)$ mirrored permutations
$\tau \in \D_n(\pi)$ with $\tau_n < \tau_{n+1}$.  Since there are $2^{n-1}$ elements in $\D_n(\pi)$, we see that there are
$\Psi(n-1,2k-1)$ elements  $\tau \in \D_n(\pi)$
with $\tau_n > \tau_{n+1}$.
\end{proof}

\begin{lemma} \label{Happy2}
Let $n$ be odd, and let $1 \leq k \leq \frac{n-1}{2}$.
 If $\pi \in P(S;n)^{2k}$ or $\pi \in P(S;n)^{2k+1}$  then there are $\Phi(n-1,2k)$ elements $\tau \in \D_n(\pi)$ with
$\tau_n \leq n$ and
 $\Psi(n-1,2k)$ elements $\tau \in \D_n(\pi)$ with $\tau_n > n$.
\end{lemma}

\begin{proof}
Suppose that $\pi = \pi_1 \pi_2 \cdots \pi_n \in P(S;n)^{{\nearrow 2k}}$. Then we know by definition  that $\pi_n= 2k$ and $\pi_{n-1} = i$ for some integer $i<2k.  $
If $\tau = \tau_1 \tau_2 \cdots \tau_n \vert  \tau_{n+1} \tau_{n+2} \cdots \tau_{2n} \in\D_n(\pi)$, then $\tau_n$  is the
$(2k)^\text{th}$ largest integer in the set
$\{ \tau_1, \tau_2, \ldots, \tau_n \}$ because $\tau$ has the same relative order as $\pi$ and $\pi_n = 2k$.
Therefore, if at least $2k$ elements of the set  $\{ \tau_1, \tau_2, \ldots, \tau_n \}$ satisfy $\tau_i \leq n$ then we can
conclude that $\tau_n \leq n$.

When $n$ is odd the construction of $\D_n(\pi)$ gives that the total number of $\tau$, where at least $2k$ of the elements from
$\{ \tau_1, \tau_2, \ldots, \tau_n \}$ were
fixed and less than or equal to $n$, is the sum
	\begin{equation}\label{nchoose2k+1}
		 \binom{n}{2k+1}+ \binom{n}{2k+3} + \cdots + \binom{n}{n-2} + \binom{n}{n}.
	 \end{equation}
Using the identity $\binom{n}{2k+1} = \binom{n-1}{2k} + \binom{n-1}{2k+1}$ we can see that the quantity in $(\ref{nchoose2k+1})$
equals
	\[  \left [\binom{n-1}{2k}+ \binom{n-1}{2k+1} \right] +\left [\binom{n-1}{2k+2}  + \binom{n-1}{2k+3} \right ]+ \cdots +   \binom{n-1}{n-1}  = \Phi(n-1,2k),
	\]    when $n$ is odd.

Suppose that $\pi = \pi_1 \pi_2 \cdots \pi_n \in P(S;n)^{{\nearrow 2k+1}}$.
Then by definition we know that $\pi_n = 2k+1$ and $\pi_{n-1} = i$ for some integer $i<2k+1.  $
Let $\tau = \tau_1 \tau_2 \cdots \tau_n \vert  \tau_{n+1} \tau_{n+2} \cdots \tau_{2n} $ be an arbitrary element of $\D_n(\pi)$.
Then $\tau_n$  is the $(2k+1)^\text{th}$ largest integer in the set
$\{ \tau_1, \tau_2, \ldots, \tau_n \}$ because $\tau$ has the same relative order as $\pi$ and $\pi_n = 2k+1$.
Therefore if at least $2k+1$ elements of the set  $\{ \tau_1, \tau_2, \ldots, \tau_n \}$ satisfy $\tau_i \leq n$ then we can
conclude that $\tau_n \leq n$.
Once again when $n$ is odd, the construction of $\D_n(\pi)$ gives that the total number of $\tau$, where at least $2k+1$ of the
elements from  $\{ \tau_1, \tau_2, \ldots, \tau_n \}$  were
fixed and less that or equal to $n$, is the sum \[ \binom{n}{2k+1}+ \binom{n}{2k+3} + \cdots + \binom{n}{n-3} + \binom{n}{n-1} =
\Phi(n-1,2k).\]
Hence the number of elements with $\tau_n \leq n$ is  $\Phi(n-1,2k)$ when $n$ is odd.   It follows that the remaining
$\Psi(n-1,2k)$  elements $\tau$ of $\D_n(\pi)$ will satisfy $\tau_n>n$.
\end{proof}

{We are now ready to enumerate the sets $ \hat{P}_D(S;n) $ and $\hat{P}_D(S\cup\{n\};n)$. }

%

\begin{proof}[Proof of Theorem \ref{key} type $D$]
Let $\pi\in\Pn$.
Assume that  $\tau = \tau_1\tau_2 \cdots \tau_n\vert \tau_{n+1}\tau_{n+2} \cdots \tau_{2n} \in \D_n$, and recall that $\D_n(\pi)$
consists of the elements of $\D_n$ which have the same relative order as $\pi$.
There are $2^{n-1}$ such elements.
Since $\tau\in\D_n(\pi)$, its first $n$ entries $\tau_1\tau_2\cdots\tau_n $ have the same relative order as
$\pi_1\pi_2\cdots\pi_n$, and just as in the type $\B_n$ case, there are two possibilities:
 \begin{itemize}
 \item either $\tau$ has the same peak set as $\pi$ so that $\tau \in \hat{P}_D(S;n)$, or
 \item $\tau$ has an additional peak at $n$, in which case $\tau \in \hat{P}_D(S\cup\{n\};n)$.
 \end{itemize}

\noindent \textbf{The two cases in which $\tau \in \hat{P}_D(S;n)$:} \\
\noindent \textbf{Case 1: }  If $\pi$ ends with a descent, i.e., $\pi_{n-1} > \pi_n$, then every $\tau \in \D_n(\pi)$ also has
$\tau_{n-1} >\tau_n$,
and thus $\tau$ is in $\hat{P}_D(S;n)$ because it cannot possibly have a peak at $n$ if it has a descent at $n-1$.
We conclude that if $\pi \in \infierno$ then all $2^{n-1}$ elements of $\D_n(\pi)$ are in
$\hat{P}_D(S;n)$. \\
\noindent \textbf{Case 2:}  If  $\pi$ ends with an ascent, $\pi_{n-1}< \pi_n$, then $\tau_{n-1} < \tau_n$ for all $\tau \in
\D_n(\pi)$ as well.
 (Recall that for any $\sigma\in \D_n$ our map into $\S_{2n}$, identifies $\sigma_{i}$ with $\sigma_{n-i+1}$ by $ \sigma_{n-i+1}
=
2n-\sigma_i+1$.)
Hence, if $\tau_n \leq n$, then $\tau_{n+1} = 2n -\tau_n+1 > \tau_n$.
In this case $\tau_{n-1} < \tau_n < \tau_{n+1}$, and $\tau$ does not have a peak at $n$. So $\tau \in \hat{P}_D(S;n)$.
Therefore we conclude that if $\pi \in \cielo $ and if $\tau \in  \D_n(\pi)$ satisfies $\tau_n \leq
n$ then $\tau$ is an
element of $\hat{P}_D(S;n)$.
By Lemma \ref{Happy1} we conclude that if $\pi \in \ATK $ then $\Phi(n-1,2k-1)$ of the elements in $\D_n(\pi)$ are in
$\hat{P}_D(S;n)$. \\
\noindent \textbf{The only case in which $\tau \in \hat{P}_D(S \cup \{n \};n)$}: \\
If $\pi \in \cielo$ and $\tau \in \D_n(\pi)$ is such that $\tau_n >n$, then $\tau$ must satisfy $\tau_{n-1} < \tau_n >
\tau_{n+1}$ because $\tau_{n+1} = 2n- \tau_n +1 <n$.
Therefore $\tau$ is an element of $\hat{P}_D(S\cup\{n\};n)$.

We have shown if $\pi$ is in $\infierno$, then all $2^{n-1}$ elements $\D_n(\pi)$ are in $\hat{P}_C(S;n)$.
 Lemma \ref{Happy1} showed when $n$ is even and  $\pi \in P(S;n)^{{\nearrow 2k}}$ or $ \pi \in
P(S;n)^{{\nearrow 2k-1}}$, then $ \Phi(n,2k-1)$ of
the elements of $\D_n(\pi)$ are in $\hat{P}_D(S;n)$.
 Thus we conclude when $n$ is even, the cardinality of $\hat{P}_D(S;n)$ is given by the formula
\[\lvert  \hat{P}_D(S;n)  \rvert=  \displaystyle\sum_{k=1}^{n} \big( \lvert  P(S;n)^{k-1} \rvert +\lvert  P(S;n)^{2k}
\rvert \big)  \cdot  \Phi(n,2k-1) + \lvert  \infierno\rvert\cdot 2^{n-1}.\]
 Lemma \ref{Happy1}  also showed if $\pi \in P(S;n)^{{\nearrow 2k}}$ or $\pi \in P(S;n)^{{\nearrow
2k-1}}$  then $\Psi(n-1,2k-1)$ elements from
$\D_n(\pi)$ are in the set  $\hat{P}_D(S\cup\{n\};n)$, and thus
\[\lvert  \hat{P}_D(S\cup\{n\};n)\rvert=\displaystyle\sum_{k=1}^{\frac{n}{2}}\big(\lvert  P(S;n)^{{\nearrow
2k-1}}\rvert+\lvert
P(S;n)^{{\nearrow 2k}}\rvert\big) \cdot \Psi(n-1,2k-1)\] when $n$ is even.

 Lemma \ref{Happy2} showed when $n$ is odd and  $\pi \in P(S;n)^{{\nearrow 2k}}$ or $ \pi \in
P(S;n)^{{\nearrow 2k-1}}$, then $ \Phi(n-1,2k)$ of the
elements of $\D_n(\pi)$ are in $\hat{P}_D(S;n)$.
 Thus we conclude that when $n$ is odd, the cardinality of $\hat{P}_D(S;n)$ is given by the formula
\begin{align*}  \lvert  \hat{P}_D(S;n)\rvert & =\displaystyle\sum_{k=1}^{ \frac{n-1}{2}} \big(\lvert
P(S;n)^{{\nearrow 2k+1}}\rvert+\lvert
P(S;n)^{{\nearrow 2k}}\rvert\big)\cdot \Phi(n-1,2k)+\lvert  \infierno\rvert \cdot 2^{n-1}.  \end{align*}
 Lemma \ref{Happy2}  also showed if $\pi \in P(S;n)^{{\nearrow 2k}}$ or $\pi \in P(S;n)^{{\nearrow 2k+1}}$  then $\Psi(n-1,2k)$ elements from
$\D_n(\pi)$ are in the set  $\hat{P}_D(S\cup\{n\};n)$, and thus
\[  \lvert  \hat{P}_D(S\cup\{n\};n)\rvert  =\displaystyle\sum_{k=1}^{\frac{n-1}{2}}\big(\lvert  P(S;n)^{{\nearrow 2k+1}}\rvert+\lvert
P(S;n)^{{\nearrow 2k}}\rvert\big)\cdot \Psi(n-1,2k) \] when $n$ is odd. This
proves the formula for the cardinality of $ \hat{P}_D(S\cup\{n\};n) $.
\end{proof}

\subsection{Special case: empty peak set in type $C$ and $D$}\label{sec:emptypeakset}
{In this section we consider the special case of $S = \emptyset$ in types $C_n$ and $D_n$. }

\begin{proposition}\label{cardC}
Let $n\geq 2$ and $m\geq 4$, then
	\begin{enumerate}[(I)]
		\item \label{emptysetC}$|  \hat{P}_{C}(\emptyset;n)|=\frac{3^n+1}{2}$
		\item \label{emptysetD}$ \lvert \hat{P}_D(\emptyset;m)\rvert = \frac{3^{m}}{4}+\frac{(-1)^{m}}{4}+\frac{1}{2}$.
	\end{enumerate}
\end{proposition}

{Proposition \ref{cardC} \eqref{emptysetC} was originally proved by Castro-Velez et al. \cite[Theorem 2.4]{CV14} in type $B_n$. It can also be proved as a corollary to Theorem \ref{key} as we show in the appendix.
However the proof given here is a combinatorial argument involving ternary sequences (in the letters $A,B$ and $C$) with
an even number of $B$'s that restricts naturally to a proof of a similar result involving the mirrored permutations with no peaks
in type $D_n$ as well.  }

The integer sequence given by Theorem \ref{cardC} \eqref{emptysetC} is sequence \textcolor{blue}{\href{https://oeis.org/A007051}{A007051}} in
Sloane's OEIS after the first three iterations \cite[A007051]{OEIS}.
Let $\TB_n$ denote the set of ternary sequences (in the letters $A,B$ and $C$) of length $n$ with an even number of $B$'s.
It is noted on Sloane's OEIS that $\frac{3^n+1}{2}$ counts all such sequences.

\begin{proof}[Proof of Proposition \ref{cardC} \eqref{emptysetC}]
To prove that $\lvert \hat{P}_C(\emptyset;n) \rvert = \lvert \TB_n \rvert = \frac{3^n+1}{2}$ we prove there is a bijection
between
the sets $\TB_n$ and $ \hat{P}_C(\emptyset;n)$.

Every permutation $\pi \in \hat{P}_C(\emptyset; n)$ has the following form
$\pi=\pi_A \pi_B \pi_C|\overline{\pi_C\pi_B\pi_A}$ where $\pi_A$ is a sequence of numbers in descending order and each $\pi_i \in
\pi_A$ is greater than $n$,
$\pi_B$ is a sequence of numbers in descending order and each $\pi_i \in \pi_B$ is less than or equal to $n$, and $\pi_C$ is a
sequence of numbers in ascending order and
each $\pi_i \in \pi_C $ is less
than or equal to $n$. Note that the mirror image $\overline{\pi_C\pi_B\pi_A}$ is determined uniquely by $\pi_A\pi_B\pi_C$, so to
condense notation in this proof we will refrain from writing it.
It is possible for at most two of the parts $\pi_A, \pi_B$, or $\pi_C$ to be empty. Moreover, there is always a choice of whether
to include the minimum element of the subpermutation
$\pi_B\pi_C$ as the last element in $\pi_B$ or the first element in $\pi_C$.
We always choose to make the length of $\pi_B$ even by including/excluding this minimum element depending on the parity of
$\pi_B$.

More precisely, let  $\pi = \pi_A \pi_B \pi_C \in \hat{P}_C(\emptyset; n)$ where \[ \pi_A=[\pi_1>\cdots > \pi_k],
\pi_B=[\pi_{k+1}>\cdots > \pi_{k+j}] \text{ and } \pi_C=[\pi_{k+j+1}<\cdots < \pi_n].\]

Define a set map $\Delta: \hat{P}_C(\emptyset; n) \rightarrow \TB_n$
by assigning a ternary sequence $\Delta(\pi) = x$ in $\TB_n$ to each element $\pi \in  \hat{P}_C(\emptyset; n)$ by setting
\[  \Delta(\pi)_i= x_i = \begin{cases}  A & \text{if } i \in \{2n-\pi_1+1, \ldots, 2n-\pi_k+1\} \\ B & \text{if } i \in
\{\pi_{k+1},\ldots, \pi_{k+j}\} \  \\C & \text{if } i \in \{ \pi_{k +j+1}, \ldots, \pi_{n}\}. \end{cases} \]
Note that there is an even number of $B'$s by the way we defined $\pi_B$. Hence $\Delta(\pi) =x \in \TB_n$.

We can also define a set map $\Theta : \TB_n \rightarrow \hat{P}_C(\emptyset; n) $ by reversing this process.
That is to say, given a ternary sequence $x=x_1x_2\cdots x_n$ in $\TB_n$ define $\mathcal{A}, \mathcal{B},$ and $\mathcal{C}$ to
be following three sets:
 \[ \mathcal{A}  =   \{ 1 \leq i \leq n: x_i = A \}, \ \ \mathcal{B} =  \{ 1 \leq i \leq n : x_i=B \} ,    \text{ and } \
\mathcal{C} =   \{ 1 \leq i \leq n: x_i = C \} . \]
List the elements of $\mathcal{A}$ and $\mathcal{C}$ in ascending order and $\mathcal{B}$ in descending order:
  \[  \mathcal{A} =  [  a_1 < a_2 < \cdots < a_k  ], \ \ \mathcal{B} = [ b_{k+1} >  b_{k+2} > \cdots >  b_{k+j}  ] , \text{ and }
  \mathcal{C}= [ c_{k+j+1} < c_{k+j+2} < \cdots < c_{n} ].  \]
Then define $\Theta(x) = \pi$ where \[ \pi_i = \begin{cases} 2n-a_i+1 & \text{if } 1 \leq i \leq k \\ b_{i} & \text{if }  k+1
\leq
i \leq k+j \\ c_{i} & \text{if } k+j+1 \leq i \leq n . \end{cases}  \]
 Notice that after $\pi_i$ is determined for $1\leq i \leq n$ then the rest of $\pi$ is determined.

To show $\Theta \circ \Delta = Id$ let $\pi=\pi_A\pi_B\pi_C \in \hat{P}_C(\emptyset; n)$ where
	\[ \pi_A=[\pi_1>\cdots > \pi_k], \pi_B=[\pi_{k+1}>\ldots > \pi_{k+j}] \text{ and } \pi_C=[\pi_{k+j+1}<\cdots < \pi_n],\]
and set  $\sigma=\Theta(\Delta(\pi))=\sigma_1\cdots \sigma_n$. Then $\Delta(\pi)_i=x_i=A$ for $i \in \{ 2n-\pi_1+1, \cdots,
2n-\pi_k+1\}$, so the list $\mathcal{A}=[2n-\pi_1+1 <\cdots< 2n-\pi_k+1]$.
By definition of $\Theta$ we get $\sigma_i = 2n-(2n-\pi_i+1)+1$ for $1\leq i \leq k$, thus $\sigma_i=\pi_i$ for $1\leq i \leq k$.

Similarly, $\Delta(\pi)_i=x_i=B$ for $i \in \{\pi_{k+1},\cdots, \pi_{k+j}\}$, thus $\mathcal{B}=[\pi_{k+1}>\ldots>\pi_{k+j}]$. By
definition of $\Theta$ we get $\sigma_i=\pi_i$ for $k+1 \leq i \leq k+j$.
Finally, $\Delta(\pi)_i=x_i=C$ for $i\in \{\pi_{k+j+1}, \ldots, \pi_{n}\}$, thus $\mathcal{C}=[\pi_{k+j+1}<\ldots<\pi_{n}]$.
By definition of $\Theta$ we see that $\sigma_i=\pi_i$ for $k+j+1 \leq i\leq n$. Therefore $\sigma_i = \pi_i$ for $1\leq i \leq
n$, which implies $\Theta(\Delta(\pi))=\sigma=\pi$ for all $\pi \in \hat{P}_C(\emptyset; n)$. A similar argument shows
$\Delta(\Theta(x))=x$ for all $x \in T_n$.
 \end{proof}


The integer sequence given by Proposition \ref{cardC} \eqref{emptysetD} is sequence \textcolor{blue}{\href{http://oeis.org/A122983}{A122983}}
in Sloane's OEIS after the first three iterations \cite[A122983]{OEIS}.
To prove this result we let $T_n$ denote the set of ternary sequences (in the letters $A,B$ and $C$) of length $n$ with an even
number of $A$'s and $B$'s.
It is noted on Sloane's OEIS that $\frac{3^{n}}{4}+\frac{(-1)^{n}}{4}+\frac{1}{2}$ counts all such sequences.
In the following proof we construct a bijection from $T_n$ to $\hat{P}_D(\emptyset;n)$ by using the maps $\Delta$ and $\Theta$,
similar to the proof of Proposition \ref{cardC}  \eqref{emptysetC}.

\begin{proof}[Proof of Proposition \ref{cardC} \eqref{emptysetD}]
 Every permutation $\pi \in \hat{P}_D(\emptyset; n)$ has the following form
$\pi=\pi_A \pi_B \pi_C|\overline{\pi_{C}}\overline{\pi_{B}}\overline{\pi_A}$ where $\pi_A$ is a sequence of numbers in descending
order and each $\pi_i \in \pi_A$ is greater than $n$,
 $\pi_B$ is a sequence of numbers in descending order and each $\pi_i \in \pi_B$ is less than or equal to $n$, and $\pi_C$ is a
sequence of numbers in ascending order and
each $\pi_i \in \pi_C $ is less
than or equal to $n$.  Since the mirror image $\overline{\pi_C\pi_B\pi_A}$ is determined by $\pi_A\pi_B\pi_C$ we will simply write
$\pi=\pi_A\pi_B\pi_C$ in the rest of this proof.
Note that it is possible for at most two of the parts $\pi_A,\pi_B,$ or $\pi_C$ to be empty.

The length of $\pi_A$ is even since every element $\pi$ in $\hat{P}_D(\emptyset; n)$ has an even number of entries in $\pi_1 \pi_2
\cdots \pi_n$ that are greater than $n$. Moreover, there is always a choice of whether to include the minimum element of the
subpermutation
$\pi_B\pi_C$ as the last element in $\pi_B$ or the first element in $\pi_C$.
We always choose to make the length of $\pi_B$ even by including/excluding this minimum element depending on the parity of
$\pi_B$. More precisely, let
$\pi = \pi_A \pi_B \pi_C \in \hat{P}_D(\emptyset; n)$ where \[ \pi_A=[\pi_1>\cdots > \pi_k], \pi_B=[\pi_{k+1}>\cdots > \pi_{k+j}]
\text{ and } \pi_C=[\pi_{k+j+1}<\cdots < \pi_n].\]

Define a set map $\Delta: \hat{P}_D(\emptyset; n) \rightarrow \TB_n$ by assigning a ternary sequence $\Delta(\pi) = x$ in $\TB_n$
to each element $\pi \in  \hat{P}_D(\emptyset; n)$ by setting
\[  \Delta(\pi)_i= x_i = \begin{cases}  A & \text{if } i \in \{2n-\pi_1+1, \ldots, 2n-\pi_k+1\} \\ B & \text{if } i \in
\{\pi_{k+1},\ldots, \pi_{k+j}\} \  \\C & \text{if } i \in \{ \pi_{k +j+1}, \ldots, \pi_{n}\}. \end{cases} \]
Note that there is an even number of $A$'s and $B$'s by the way we defined $\pi_A$ and $\pi_B$. Hence $\Delta(\pi) =x \in \TB_n$.

We can also define a set map $\Theta : \TB_n \rightarrow \hat{P}_D(\emptyset; n) $ by reversing this process.
That is to say, given a ternary sequence $x=x_1x_2\cdots x_n$ in $\TB_n$ define $\mathcal{A}, \mathcal{B},$ and $\mathcal{C}$ to
be following three sets:
 \[ \mathcal{A}  =   \{ 1 \leq i \leq n: x_i = A \}, \ \ \mathcal{B} =  \{ 1 \leq i \leq n : x_i=B \} , \ \   \text{ and } \
\mathcal{C} =   \{ 1 \leq i \leq n: x_i = C \} . \]
List the elements of $\mathcal{A}$ and $\mathcal{C}$ in ascending order and $\mathcal{B}$ in descending order as follows:
  \[  \mathcal{A} =  [  a_1 < a_2 < \cdots < a_k  ], \ \  \mathcal{B} = [ b_{k+1} >  b_{k+2} > \cdots >  b_{k+j}  ] , \text{ and }
\] \[   \mathcal{C}= [ c_{k+j+1} < c_{k+j+2} < \cdots < c_{n} ].  \]
Then define $\Theta(x) = \pi$ where \[ \pi_i = \begin{cases} 2n-a_i+1 & \text{ if } 1 \leq i \leq k \\ b_{i} & \text{ if }  k+1
\leq i \leq k+j \\ c_{i} & \text{ if } k+j+1 \leq i \leq n .\end{cases}  \]
 Notice that after $\pi_i$ is determined for $1\leq i \leq n$ then the rest of $\pi$ is determined.

To show $\Theta \circ \Delta = Id$ let $\pi=\pi_A\pi_B\pi_C \in \hat{P}_D(\emptyset; n)$ where
\[ \pi_A=[\pi_1>\cdots > \pi_k], \pi_B=[\pi_{k+1}>\ldots > \pi_{k+j}] \text{ and } \pi_C=[\pi_{k+j+1}<\cdots < \pi_n],\]
and set  $\sigma=\Theta(\Delta(\pi))=\sigma_1\cdots \sigma_n$. Then $\Delta(\pi)_i=x_i=A$ for $i \in \{ 2n-\pi_1+1, \cdots,
2n-\pi_k+1\}$, so the list $\mathcal{A}=[2n-\pi_1+1 <\cdots< 2n-\pi_k+1]$.
By the definition of $\Theta$ we get $\sigma_i = 2n-(2n-\pi_i+1)+1$ for $1\leq i \leq k$, thus $\sigma_i=\pi_i$ for $1\leq i \leq
k$.

Similarly, $\Delta(\pi)_i=x_i=B$ for $i \in \{\pi_{k+1},\cdots, \pi_{k+j}\}$, and thus $\mathcal{B}=[\pi_{k+1}>\ldots>\pi_{k+j}]$.
By definition of $\Theta$ we see that $\sigma_i=\pi_i$ for $k+1 \leq i \leq k+j$. Finally, $\Delta(\pi)_i=x_i=C$ for $i\in
\{\pi_{i+k+1}, \ldots, \pi_{n}\}$, thus $\mathcal{C}=[\pi_{k+j+1}<\ldots<\pi_{n}]$. By definition of $\Theta$ we get
$\sigma_i=\pi_i$ for $k+j+1 \leq i\leq n$. Therefore $\sigma_i = \pi_i$ for $1\leq i \leq n$, which implies
$\Theta(\Delta(\pi))=\sigma=\pi$ for all $\pi \in \hat{P}_D(\emptyset; n)$. A similar argument shows $\Delta(\Theta(x))=x$ for all
$x \in T_n$.
\end{proof}

We will illustrate the bijection between $\Delta$ and $\Theta$, described in the proof of Proposition \ref{cardC}, with the following example.
\begin{example}
\underline{Type $C$}:
Consider the permutation $\pi \in \B_{10}$ where
	\[\pi = 20 \ 18 \ 13 \ 10 \ 9 \ 7 \ 4 \ 2 \  5 \ 6 \ | \ 15 \ 16 \ 19 \ 17 \  14 \ 12 \ 11 \ 8 \ 3 \ 1.\]
Let $\Delta(\pi)=x \in T_n$. Since
\[ \pi_A = 20 \ 18 \ 13  , \ \  \pi_B = 10 \ 9 \ 7 \ 4  , \text{ and } \pi_C  = 2 \ 5 \ 6,\]
then $x_i=A$ for $i \in \{1,3,8\}$, $x_i = B$ for $i\in \{4,7,9, 10\}$, and $x_i=C$ for $i \in \{2, 5,6\}$. Thus  $ \Delta(\pi) =
x = ACABCCBABB$.

Consider $ \Theta(\Delta(\pi))\in \hat{P}_C(\emptyset;10)$. Since $ \Delta(\pi) = x = ACABCCBABB$ then the lists $\mathcal{A},
\mathcal{B}$ and $\mathcal{C}$ are defined as
 \[ \mathcal{A}=[1<3<8], \ \ \mathcal{B} = [ 10>9 > 7 > 4   ], \text{ and } \mathcal{C} = [ 2<5< 6 ].\]
 Using the definition of $\Theta$ we get
 \[ \Theta(\Delta(\pi))=\Theta(x) = 20 \ 18 \ 13 \ 10 \ 9 \ 7 \ 4 \ 2 \  5 \ 6 \ | \ 15 \ 16 \ 19 \ 17 \  14 \ 12 \ 11 \ 8 \ 3 \
1= \pi.\qedhere\]

\noindent
\underline{Type $D$}:
Consider the permutation $\pi \in \D_{10}$ where
	\[\pi = 20 \ 18 \ 13 \ 11 \ 9 \ 7 \ 4 \ 2 \  5 \ 6 \ | \ 15 \ 16 \ 19 \ 17 \  14 \ 12 \ 10 \ 8 \ 3 \ 1.\]
Let $\Delta(\pi)=x \in T_n$. Since
\[ \pi_A = 20 \ 18 \ 13 \ 11 , \ \  \pi_B = 9 \ 7 \ 4 \ 2 , \text{ and } \pi_C  = 5 \ 6,\]
then $x_i=A$ for $i \in \{1,3,8,10\}$, $x_i = B$ for $i\in \{2,4,7,9\}$, and $x_i=C$ for $i \in \{5,6\}$. Thus  \[ \Delta(\pi) =
x
= ABABCCBABA. \]
Consider $ \Theta(\Delta(\pi))\in \hat{P}_D(\emptyset;10)$. Since $ \Delta(\pi) = x = ABABCCBABA$ then the lists $\mathcal{A},
\mathcal{B}$ and $\mathcal{C}$ are defined as
 \[ \mathcal{A}=[1<3<8<10], \ \ \mathcal{B} = [ 9 > 7 > 4 > 2  ], \text{ and } \mathcal{C} = [ 5< 6 ].\]
 Using the definition of $\Theta$ we get
 \[ \Theta(\Delta(\pi))=\Theta(x) = 20 \ 18 \ 13 \ 11 \ 9 \ 7 \ 4 \ 2 \  5 \ 6 \ | \ 15 \ 16 \ 19 \ 17 \  14 \ 12 \ 10 \ 8 \ 3 \
1= \pi. \]
\end{example}

\section{Questions and future work}

We end this paper with a few questions of interest.
We suspect that the sets we call pattern bundles have appeared elsewhere in the literature on Coxeter groups,
but we do not know of such a reference.   (Note that the pattern bundles are the fibers of an order-preserving flattening map from
$\B_n$ to $\S_n$
 that differs from the usual $2^n$ to $1$ projection of signed permutations to $\S_n$ which forgets the negative signs.)
If these sets have not been studied before, then our first question is:

\begin{problem}
Can the pattern bundles of types $C_n$ and $D_n$ be used to study other permutation statistics (such as descent sets for
instance)?
\end{problem}

We can also ask whether these techniques can be applied to study other groups having \emph{suitably nice} embeddings into $\S_N$,
and whether the peak set of the image encodes any information about the embedded group.
\begin{problem}
Can the methods used in this paper be applied to study peak sets of groups such as the dihedral groups or Coxeter groups of
exceptional type by embedding them into $\S_{N}$ for some $N$?
\end{problem}

We provide recursive formulas for the quantities $|\hat{P}_C(S;n)|$ and  $|\hat{P}_D(S;n)|$ in Theorem \ref{key} that can be used to find closed formulas
for any particular choice of peak set $S$.  Several of the special cases we consider in this paper give closed formulas for
integer sequences appearing on \emph{Sloane's Online Encyclopedia of Integer Sequences (OEIS)} \cite{OEIS}.
Hence we believe the following would be an interesting undergraduate student research project.
\begin{problem}[Undergraduate Student Research Project]
Can one compute closed formulas for some families of peak sets and analyze which of these appear on the OEIS?
\end{problem}
This leads us to our final question:
\begin{problem}
   Can one discover closed combinatorial formulas for $|\hat{P}_C(S;n)|$ and  $|\hat{P}_D(S;n)|$ in general?
\end{problem}

\section{Acknowledgements}
The authors would like to thank the \emph{Underrepresented Students in Topology and Algebra Symposium} (USTARS);
if not for our chance encounter at USTARS this collaboration may not have materialized! We also thank Sara Billey, Christophe
Hohlweg, and Bruce Sagan for helpful conversations about this paper.
Pamela E. Harris gratefully acknowledges travel support from the Photonics Research Center and the Mathematical Sciences Center of
Excellence at the United States Military Academy.

\appendix
\section{Alternate proofs}
In this appendix we provide an alternate proof of Theorem \ref{cardC} \eqref{emptysetC} by proving it as a corollary to Theorem \ref{key}.  Our
alternate proof will use  the following curious identity.
\begin{lemma}\label{curious}If $n\geq 2,$ then $\displaystyle\sum_{k=2}^{n+1}2^{k-2}\Phi(n,k-1)+2^n=3^n.$
\end{lemma}

\begin{proof}We proceed by induction.
Observe that when $n=2$, we have that
\[\displaystyle\sum_{k=2}^{3}2^{k-2}\Phi(2,k-1)+4=2^0\Phi(2,1)+2^1\Phi(2,2)+4=3+2+4=9=3^2.\]
We assume that $\displaystyle\sum_{k=2}^{n+1}2^{k-2}\Phi(n,k-1)+2^n=3^n.$

Observe that using the identity $\Phi(n+1,k-1)=\Phi(n,k-1)+\Phi(n,k-2)$ and our induction hypothesis we have
\begin{align*}
\displaystyle\sum_{k=2}^{n+2}2^{k-2}\Phi(n+1,k-1)+2^{n+1}&=\displaystyle\sum_{k=2}^{n+2}2^{k-2}\Phi(n,k-1)+\displaystyle\sum_{k=2}
^{n+2}2^{k-2}\Phi(n,k-2)+2^{n+1}\\
&=\displaystyle\sum_{k=2}^{n+1}2^{k-2}\Phi(n,k-1)+2\displaystyle\sum_{k=1}^{n+1}2^{k-2}\Phi(n,k-1)+2^{n+1}\\
&=3\displaystyle\sum_{k=2}^{n+1}2^{k-2}\Phi(n,k-1)+2^n+2^{n+1}\\
&=3^{n+1}.\end{align*}
\end{proof}

\noindent Now that we have established Lemma \ref{curious} we can prove Theorem \ref{cardC} \eqref{emptysetC} as a corollary of Theorem \ref{key}.
\begin{proof}[Alternate proof of Theorem \ref{cardC}]
We proceed by induction on $n$. If $n=2$ we have previously computed that $\lvert  \hat{P}_C(\emptyset,2)\rvert=5$, which is the
same as $\frac{3^2+1}{2}=5$.
Assume that for any $k\leq n$ $\lvert  \hat{P}_C(\emptyset,k)\rvert=\frac{3^k+1}{2}$. We want to show that the formula holds
for $k=n+1$. To do so we recall that by Theorem \ref{key} and Lemma \ref{powerof2} we know that
\[\lvert  \hat{P}_C(\emptyset;n)\rvert=\displaystyle\sum_{k=1}^{n}2^{k-2}\Phi(n,k)+2^n.\]
Hence
\begin{align*}
\lvert  \hat{P}_C(\emptyset;n+1)\rvert&=\displaystyle\sum_{k=1}^{n+1}2^{k-2}\Phi(n+1,k)+2^{n+1}\\
&=2^{n-1}+2^{n+1}+\displaystyle\sum_{k=1}^{n}2^{k-2}(\Phi(n,k-1)+\Phi(n,k))\\
&=2^{n-1}+2^{n}+\left(\frac{3^n+1}{2}\right)+\displaystyle\sum_{k=1}^{n}2^{k-2}\Phi(n,k-1).
\end{align*}
Observe that using Lemma
\ref{curious}\[\displaystyle\sum_{k=1}^{n}2^{k-2}\Phi(n,k-1)+2^{n-1}+2^{n}=\displaystyle\sum_{k=1}^{n+1}2^{k-2}\Phi(n,
k-1)+2^n=3^n.\]
Therefore
\[\lvert  \hat{P}_C(\emptyset;n+1)\rvert=\frac{3^n+1}{2}+3^n=\frac{3^{n+1}}{2}.\]
\end{proof}

We end this appendix with one more problem to consider.

\begin{problem}
Can one provide an analogous proof of Theorem Theorem \ref{cardC} \eqref{emptysetD} following a similar argument as that of the alternate proof of
Theorem \ref{cardC} in the appendix?
This does not seem to be a difficult problem, but we were unable to provide an analogous proof.
\end{problem}


\begin{thebibliography}{10}

\bibitem{ABN04}
M.~Aguiar, N.~Bergeron, and K.~Nyman, The peak algebra and the descent algebras
  of types {B} and {D}, {\em Trans. Amer. Math. Soc.} {\bf 356} (2004),
  2781--2824.

\bibitem{ABS06}
M.~Aguiar, N.~Bergeron, and F.~Sottile, Combinatorial {H}opf algebras and
  generalized {D}ehn-{S}ommerville relations, {\em Compos. Math.} {\bf 142}
  (2006), 1--30.

\bibitem{ANO06}
M.~Aguiar, K.~Nyman, and R.~Orellana, New results on the peak algebra, {\em J.
  Algebraic Combin.} {\bf 23} (2006), 149--188.

\bibitem{BH06}
N.~Bergeron and C.~Hohlweg, Colored peak algebras and {H}opf algebras, {\em J.
  Algebraic Combin.} {\bf 24} (2006), 299--330.

\bibitem{BMSV00}
N.~Bergeron, S.~Mykytiuk, F.~Sottile, and S.~van Willigenburg, Noncommutative
  {P}ieri operators on posets. {I}n memory of {G}ian-{C}arlo {R}ota, {\em J.
  Combin. Theory Ser. A} {\bf 91} (2000), 84--110.

\bibitem{BMSV02}
N.~Bergeron, S.~Mykytiuk, F.~Sottile, and S.~van Willigenburg, Shifted
  quasi-symmetric functions and the {H}opf algebra of peak functions, {\em
  Discrete Math.} {\bf 246} (2002), 57--66.

\bibitem{BS02}
N.~Bergeron and F.~Sottile, Skew {S}chubert functions and the {P}ieri formula
  for flag manifolds, {\em Trans. Amer. Math. Soc.} {\bf 354} (2002),
  651--673.

\bibitem{BHV03}
L.~J. Billera, S.K. Hsia, and S.~van Willigenburg, Peak quasisymmetric
  functions and {E}ulerian enumeration, {\em Adv. Math.} {\bf 176} (2003),
  248--276.

\bibitem{BBPS14}
S.~Billey, K.~Burdzy, S.~Pal, and B.~Sagan, On meteors, earthworms and {WIMP}s, preprint, \url{http://arxiv.org/abs/1308.2183}
(2013).

\bibitem{BBS13}
S.~Billey, K.~Burdzy, and B.~Sagan, Permutations with given peak set, {\em J.
  of Integer Seq.} {\bf 16} (2013).

\bibitem{BFT14}
S.~Billey, M.~Fahrbach, and A.~Talmage, Coefficients and roots of peak
  polynomials, preprint, \url{http://arxiv.org/abs/1410.8506}  (2014).

\bibitem{BH95}
S.~Billey and M.~Haiman, Schubert polynomials for the classical groups, {\em J.
  Amer. Math. Soc.} {\bf 8} (1995), 443--482.

\bibitem{BL00}
S.~Billey and V.~Lakshmibai, {\em Singular Loci of {S}chubert Varieties},
  Birkhauser, 2000.

\bibitem{BB00}
A.~Bj\"orner and F.~Brenti, {\em Combinatorics of {C}oxeter Groups}, Springer,
  2000.

\bibitem{CV14}
F.~Castro-Velez, A.~Diaz-Lopez, R.~Orellana, J.~Pastrana, and R.~Zevallos,
  Number of permutations with same peak set for signed permutations, preprint, \url{http://arxiv.org/abs/1308.6621}  (2014).

\bibitem{DLHIPL}
A.~Diaz-Lopez, P.~E. Harris, E.~Insko, D.~Perez-Lavin,
Peak Sets of Classical Coxeter Groups, preprint, \url{http://arxiv.org/pdf/1505.04479} (2015).

\bibitem{K14}
A.~Kasraoui, The most frequent peak set in a random permutation, preprint, \url{http://arxiv.org/abs/1210.5869}  (2012).

\bibitem{N03}
K.~Nyman, The peak algebra of the symmetric group, {\em J. Algebraic Combin.}
  {\bf 17} (2003), 309--322.

\bibitem{P07}
T.~K. Petersen, Enriched {P}-partitions and peak algebras, {\em Adv. Math.}
  {\bf 209} (2007), 561--610.

\bibitem{OEIS}
N.~J.~A. Sloane, The Online Encyclopedia of Integer Sequences,
  \url{http://oeis.org}, 2010.

\bibitem{S12}
R.~Stanley, {\em Enumerative Combinatorics Volume 1}, Cambridge University
  Press, 2nd edition, 2012.

\bibitem{S97}
J.~Stembridge, Enriched {P}-partitions., {\em Trans. Amer. Math. Soc.} {\bf
  349} (1997), 763--788.

\bibitem{S78}
V.~Strehl, Enumeration of alternating permutations according to peak sets, {\em
  J. Combin. Theory Ser. A} {\bf 24} (1978), 238--240.

\end{thebibliography}
\end{document}